\documentclass[reqno]{amsart}
\usepackage{fullpage}
\usepackage{amsmath} 

\usepackage{amsthm}
\usepackage{amssymb}
\usepackage{graphics}
\usepackage{latexsym}
\usepackage{comment}

\numberwithin{equation}{section}
\newtheorem{thm}{Theorem}[section]

\newtheorem{lem}[thm]{Lemma}
\newtheorem{prop}[thm]{Proposition}

\theoremstyle{definition}

\newtheorem{rem}{Remark}

\newcommand{\laplacian}{\Delta}
\newcommand{\R}{{\mathbb R}}

\newcommand{\C}{{\mathbb C}}
\newcommand{\LR}[1]{{\langle {#1} \rangle }}
\newcommand{\cross}{\times}
\newcommand{\e}{\varepsilon}
\newcommand{\F}{\mathcal{F}}

\newcommand{\ha}{\widehat}

\title[Well-posedness of the Klein-Gordon-Zakharov system] 
      {Well-posedness for the Cauchy problem of \\
the Klein-Gordon-Zakharov system 
in 2D}

\author[Shinya Kinoshita]{}


 \email{m12018b@math.nagoya-u.ac.jp}

\begin{document}
\maketitle

\centerline{\scshape Shinya Kinoshita}
\medskip
{\footnotesize
 \centerline{Graduate School of Mathematics, Nagoya University}
   \centerline{Chikusa-ku, Nagoya, 464-8602, Japan}
} 

\begin{abstract}
This paper is concerned with the Cauchy problem of the Klein-Gordon-Zakharov system with very low regularity initial data. 
We prove the bilinear estimates which are crucial to get the local in time well-posedness. The estimates are established by the Fourier 
restriction norm method. We utilize the nonlinear version of the classical Loomis-Whitney inequality.
\end{abstract}

\section{Introduction}
We consider the Cauchy problem of the Klein-Gordon-Zakharov system: 
\begin{equation}
 \begin{cases}
  (\partial_t^2  - \laplacian  + 1)u = -nu, \qquad (t,x) \in [-T,T] \cross \R^d, \\
  (\partial_t^2 - c^2 \laplacian )n = \laplacian |u|^2, \qquad \ \, (t,x) \in [-T,T] \cross \R^d, \\
  (u, \partial_t u, n, \partial_t n)|_{t=0} = (u_0, u_1, n_0, n_1) \\
   \qquad  \qquad  \qquad  \qquad \in H^{s+1}(\R^d) \cross H^s(\R^d) \cross H^s(\R^d) \cross H^{s-1}(\R^d), 
                                                                                                                                            \label{KGZ}
 \end{cases}  
\end{equation}
where $u, n$ are real valued functions, $0< c<1$.
As a physical model, \eqref{KGZ} describes the interaction of the Langmuir wave and the ion acoustic wave in a plasma. 
The condition $0<c<1$, which plays an important role in the paper, comes from 
a physical phenomenon. 
See Bellan \cite{Bellan}, Masmoudi and Nakanishi \cite{MN0}. 
There are some works on the Cauchy problem of \eqref{KGZ} in low regularity Sobolev spaces. 
For $3$D, Ozawa, Tsutaya and Tsutsumi \cite{OTT2}  proved that \eqref{KGZ} is globally well-posed in the energy space $H^1(\R^3) \cross L^2(\R^3) \cross L^2(\R^3) \cross \dot{H}^{-1}(\R^3)$. 
In the case of $c=1$, \eqref{KGZ} is very similar to the Cauchy problem of the following quadratic derivative nonlinear wave equation.
\begin{equation}
\begin{cases}
  (\partial_t^2  - \laplacian )u = u Du, \qquad (t,x) \in [-T,T] \cross \R^3, \\
  (u, \partial_t u)|_{t=0} = (u_0, u_1) \in H^{s+1}(\R^3) \cross H^s(\R^3). \label{NW}
\end{cases}
\end{equation}
For $s>0$, the local well-posedness of \eqref{NW} was obtained from the iteration argument by using the 
usual Strichartz estimates. As opposed to that, 
Lindblad showed that \eqref{NW} is ill-posed for $s \leq0$, see \cite{L0}-\cite{L1}. In \cite{OTT2}, 
Ozawa, Tsutaya and Tsutsumi showed that the difference between the propagation speeds of the two equations in \eqref{KGZ} allows for a better result. That is, they 
applied the Fourier restriction norm method and obtained the local well-posedness of \eqref{KGZ} in the energy space, and then by the energy conservation law, they extended solutions globally in time. 
By the similar argument, Tsugawa established that \eqref{KGZ} is local well-posed in $2$D for $s \geq -1/2$. 
For $4$ and higher dimensions, I. Kato ~\cite{Ka} recently proved that \eqref{KGZ} is locally well-posed at $s=1/4$ when $d=4$ and $s=s_c+1/(d+1)$ when $d \ge 5$ where $s_c=d/2-2$ is the critical exponent of \eqref{KGZ}. 
He also proved that if the initial data are radially symmetric then the small data globally 
well-posedness can be obtained at the scaling critical regularity for $d \geq4$. 
He utilized the $U^2$, $V^2$ spaces introduced by Koch-Tataru \cite{KT1}. We would like to emphasize that the above results 
hold under the condition $0<c<1$. Our aim in this paper is to get the local well-posedness of \eqref{KGZ} at very low regularity $s$ in 
$2$ dimensions. Hereafter we assume $d=2$. 

By the transformation $u_{\pm} := \omega_1 u \pm i\partial_t u, n_{\pm} := n \pm i(c\omega_1)^{-1}\partial_t n, 
\omega_1:=(1-\laplacian )^{1/2}, \omega := (-\laplacian )^{1/2}$, \eqref{KGZ} can be written as follows:
\begin{equation}
 \begin{cases}
  (i\partial_t \mp \omega_1) u_{\pm} 
    = \pm (1/4)(n_+ + n_-)(\omega_1^{-1}u_+ + \omega_1^{-1}u_-), 
               \\
  (i\partial_t \mp c\omega_1 )n_{\pm} 
    = \pm (4c)^{-1}\omega^2 \omega_1^{-1} | \omega_1^{-1} u_+ + \omega_1^{-1} u_-|^2 \mp c (2 \omega_1)^{-1} 
(n_+ + n_-) , \\
  (u_{\pm}, n_{\pm})|_{t=0} = (u_{\pm 0}, n_{\pm 0}) 
                                \in H^s(\R^2) \cross H^s(\R^2). 
                                                                                                                                            \label{KGZ'}
 \end{cases}  
\end{equation}
We state our main result.  
\begin{thm}  \label{mth}
Let $0<c<1$ and $-3/4< s < 0$. Then \eqref{KGZ'} is locally well-posed in $H^{s}(\R^2) \cross H^{s}(\R^2)$.
\end{thm}
\begin{rem}\label{Rem-homo}
We can show the local well-posedness of \eqref{KGZ} in $H^{s+1} \cross H^s 
\cross {\dot{H}}^s \cross {\dot{H}}^{s-1}$ instead of $H^{s+1} \cross H^s \cross H^s \cross H^{s-1}$. 
In that case, by the transformation $n_{\pm} := n \pm i(c\omega)^{-1}\partial_t n$ instead of 
$n_{\pm} := n \pm i(c\omega_1)^{-1}\partial_t n$, \eqref{KGZ} is written as follows:
\begin{equation}
 \begin{cases}
  (i\partial_t \mp \omega_1) u_{\pm} 
    = \pm (1/4)(n_+ + n_-)(\omega_1^{-1}u_+ + \omega_1^{-1}u_-), \\
  (i\partial_t \mp c\omega )n_{\pm} 
    = \pm (4c)^{-1}\omega | \omega_1^{-1} u_+ + \omega_1^{-1} u_-|^2, \\
  (u_{\pm}, n_{\pm})|_{t=0} = (u_{\pm 0}, n_{\pm 0}) \in H^s(\R^2) \cross \dot{H}^s(\R^2). \label{KGZ-homo}
\end{cases}  
\end{equation}
See {\itshape Remark} \ref{rem1.2} below for the details.
\end{rem}
We make a comment on Theorem \ref{mth}. Applying the iteration argument by the usual Strichartz estimates, we get the 
local well-posedness of \eqref{KGZ'} for $-1/4 \leq s$. This suggests that if $c =1$ the minimal regularity such that the well-posedness of \eqref{KGZ'} holds seems to be $-1/4$. 
Tsugawa \cite{Tsugawa} found that if we utilize the condition $0<c<1$ in the same way as in \cite{OTT2} with 
minor modification, we can show that \eqref{KGZ'} is local well-posed only for $s \geq -1/2$. 
We can say that the known arguments is not enough to get the well-posedness for 
$s < -1/2$ which is the most difficult case. To overcome this, we employ a new estimate which was introduced in \cite{BKW}, \cite{BHT} 
and applied to the Zakharov system in \cite{BHHT} and \cite{BH}. 
See Proposition \ref{prop2.7} below. 
The Zakharov system consists of a wave equation and 
a Schr\"{o}dinger equation:
\begin{align}
 \begin{cases}
  (i\partial_{t} + \laplacian )u = nu, \ \ \ \ \ \ \ \ \ \ \ \ \ \ \ \ \  (t, x) \in \R \cross \R^d,  \\
  (\partial_{t}^{2} - \laplacian )n = \laplacian |u|^2, \ \ \ \ \ \ \ \ \ \ \ \ \ \ (t, x) \in \R \cross \R^d.\label{Zakharov}
 \end{cases}
\end{align}
Roughly speaking, the two systems \eqref{KGZ} and \eqref{Zakharov} share two features: \\
(I). The two linear dispersive differential operators are different from each other.\\
(II). The nonlinear terms are all quadratic.\\
These similarities suggest that we 
might get the well-posedness of \eqref{KGZ'} for $s < -1/2$ in the same way as in \cite{BHHT} and \cite{BH}. 

We will prove Theorem \ref{mth} by the iteration argument in the spaces 
$X^{s,\,b}_\pm(\R^3)  \cross X^{s,\,b}_{\pm,\,c}(\R^3)$. This spaces are defined 
as follows:

Let $0 < c \leq 1$ and $N$, $L \geq 1$ be dyadic numbers. $\chi_\Omega$ denotes the characteristic function of a set $\Omega$. 
We define the dyadic decompositions of $\R^3$.
\begin{equation*}
K^{\pm, c}_{N,L}  :=\{ (\tau, \xi) \in \R^3 | N \leq \LR{\xi} \leq 2N, L \leq \LR{\tau \pm c|\xi|} \leq 2L \}.
\end{equation*}
By using $K^{\pm, c}_{N,L}$, we introduce the solution spaces. 
Let $s, b \in \R$. We define $X^{s,\,b}_{\pm,\,c}(\R^3)$ as follows:
\begin{align*}
&  X^{s,\,b}_{\pm,\,c}(\R^3) :=\{ f \in \mathcal{S}'(\R^3) \ | \ \|f \|_{X^{s,\,b}_{\pm,\,c}} < \infty \},\\
& \|f  \|_{X^{s,\, b}_{\pm,\, c}} := \left( \sum_{N,\, L} N^{2s} L^{2b} \|  \chi_{K^{\pm, c}_{N,L}} \ha{f} \|^2_{L_{x, t}^{2}} \right)^{1/2}.
\end{align*}
Here $\ha{f}$ denotes the Fourier transform of $f$ in space and time. 
For convenience, we define $\ha{X}^{s,\,b}_{\pm,\,c}$ which is the Fourier transform of ${X}^{s,\,b}_{\pm,\,c}$.
\begin{equation*}
\ha{X}^{s,\,b}_{\pm,\,c}(\R^3) :=\{ f \in \mathcal{S}'(\R^3) \ | \ \|f \|_{\ha{X}^{s,\,b}_{\pm,\,c}} < \infty \}, 
\ \  \ \ \, \,  \|\ha{f} \|_{\ha{X}^{s,\,b}_{\pm,\,c}} := \|f \|_{X^{s,\,b}_{\pm,\,c}}.
\end{equation*}
We denote $K^{\pm, 1}_{N,L}$ by $K^{\pm}_{N,L}$, the space $X^{s,\,b}_{\pm,\,1}$ by $X^{s,\,b}_{\pm}$ and 
its norm by $\| \cdot \|_{X^{s,\,b}_{\pm}}$, and also 
$\ha{X}^{s,\,b}_{\pm,\,1}$ by $\ha{X}^{s,\,b}_{\pm}$ and 
its norm by $\| \cdot \|_{\ha{X}^{s,\,b}_{\pm}}$.

The key estimates to prove Theorem \ref{mth} are the following.
\begin{thm}\label{nonlinearity-estimate}
Let $0<c<1$. For any $s \in \left( -3/4, \ 0 \right)$, there exists $b \in (1/2, 1)$, $\e >0$ and $C>0$ which depend on $c$ such that
\begin{align}
\|u (\omega_1^{-1} v) \|_{X^{s,\,b-1+\e}_{\pm_2}} \quad  & \leq C \|u \|_{X^{s,\,b}_{\pm_0,\,c}} \|v \|_{X^{s,\,b}_{\pm_1}} \label{goal-1-1},\\
\|\omega_1 (( \omega_1^{-1} u ) (\omega_1^{-1} v) )\|_{X^{s,\,b-1+\e}_{\pm_0, \,c}} & \leq C \|u \|_{X^{s,\,b}_{\pm_1}} \|v \|_{X^{s,\,b}_{\pm_2}}\label{goal-2-1}.  
\end{align}
regardless of the choice of signs $\pm_j$.
\end{thm}
\begin{rem}\label{rem1.2} $ \ $\\
(1) The key bilinear estimates naturally derived from \eqref{KGZ-homo} in {\itshape Remark} \ref{Rem-homo} are slightly different from \eqref{goal-1-1}-\eqref{goal-2-1}. 
They are described as follows:
\begin{align}
\|u (\omega_1^{-1} v) \|_{X^{s,\,b-1+\e}_{\pm_2}} \quad  & \leq C \|\omega^s u \|_{{X}^{0,\,b}_{\pm_0,\,c}} \|v \|_{X^{s,\,b}_{\pm_1}},\label{der1}\\
\|\omega^{s+1} (( \omega_1^{-1} u ) (\omega_1^{-1} v) )\|_{{X}^{0,\,b-1+\e}_{\pm_0, \,c}} & \leq C \|u \|_{X^{s,\,b}_{\pm_1}} \|v \|_{X^{s,\,b}_{\pm_2}}.
\label{der2}
\end{align}
Let $s \in \left( -3/4, \ 0 \right)$. From the inequalities $\|u \|_{X^{s,\,b}_{\pm_0,\,c}} \leq \|\omega^s u \|_{{X}^{0,\,b}_{\pm_0,\,c}}$ and 
\begin{equation*}
\|\omega^{s+1} (( \omega_1^{-1} u ) (\omega_1^{-1} v) )\|_{{X}^{0,\,b-1+\e}_{\pm_0, \,c}} \leq 
\|\omega_1 (( \omega_1^{-1} u ) (\omega_1^{-1} v) )\|_{X^{s,\,b-1+\e}_{\pm_0, \,c}},
\end{equation*}
it is clear that \eqref{goal-1-1} and \eqref{goal-2-1} imply \eqref{der1} and \eqref{der2}, 
respectively.\\ 
(2) It might be natural that we use $\LR{\tau \pm c\LR{\xi}}$ instead of $\LR{\tau \pm c|\xi|}$ in the definition of $K^{\pm,c}_{N,L}.$ As was seen in \cite{OTT2}, these two weights are equivalent and therefore $X^{s,\,b}_{\pm,c}$ 
does not depend on the choice of them in the definition of $K^{\pm,c}_{N,L}.$
\end{rem}
Once Theorem \ref{nonlinearity-estimate} is verified, we can obtain Theorem \ref{mth} by the iteration argument given in \cite{GTV} and many other papers. For example, see \cite{KPV}, \cite{Tao}. Therefore we focus on the proof of Theorem 
\ref{nonlinearity-estimate} in this paper. 

Next, we show the negative result for $s<- 3/4$. 
\begin{thm}\label{not-c2}
Let $d=2$, $0<c<1$ and $s < - \frac{3}{4}$. Then for any $T>0$, the data-to-solution map 
$ (u_0, u_1, n_0, n_1) \mapsto (u, n)$ 
of \eqref{KGZ}, as a map from the unit ball in 
$ H^{s+1} \cross H^s \cross H^s \cross H^{s-1}$ to 
$C([0,T]; H^{s+1}) \cap C^1([0,T];H^s) \cross C([0,T]; H^s) \cap  
C^1 ([0,T];H^{s-1})$ fails to be $C^2$.
\end{thm}
Theorem \ref{not-c2} implies that the iteration argument, which is applied to the proof of Theorem \ref{mth}, is no longer available for the case $s < -3/4$. 

The paper is organized as follows. In Section 2, we introduce some fundamental 
estimates and property of the solution spaces as preliminary. In Section 3, we show \eqref{goal-1-1} and 
\eqref{goal-2-1} with $\pm_1 =\pm_2$ which is the easier case compared to 
$\pm_1 \not= \pm_2$. In Section 4, we prove  \eqref{goal-1-1} and 
\eqref{goal-2-1} with $\pm_1 \not=\pm_2$, and complete the proof of Theorem \ref{nonlinearity-estimate}. 
Lastly as Section 5, we show the negative result, Theorem \ref{not-c2}.
\section{Preliminaries}
In this section, we introduce some estimates which will be utilized for the proof of Theorem \ref{nonlinearity-estimate}. 
Throughout the paper, we use the following notations. 
$A\lesssim B$ means that there exists $C>0$ such that $A \leq CB.$ 
Also, $A\sim B$ means $A\lesssim B$ and $B\lesssim A.$ It should be emphasized that the signs $\lesssim$ and 
$\sim$ frequently depend on $1-c$ in the paper. 
Thus, the necessary condition of a time ingredient $T$ in \eqref{KGZ} to show Theorem \ref{mth} also depends on $1-c$. Since the aim of the paper is to show the local well-posedness, 
here we are not concerned with how the necessary condition of $T$ for Theorem \ref{mth} changes as $c$ approaches to $1$. 
Let $u=u(t,x).\ \F_t u,\ \F_x u$ denote the Fourier transform of $u$ in time, space, respectively. 
$\F_{t,\, x} u = \ha{u}$ denotes the Fourier transform of $u$ in space and time.  
We first observe that fundamental properties of $X^{s,\,b}_{\pm,\,c}$. 
A simple calculation gives the following:
\begin{equation*}
(i)  \ \ 
\overline{{X}^{s,\,b}_{\pm,\,c}} = X^{s,\,b}_{\mp,\,c} , \qquad 
(ii)  \ \ (X^{s,\,b}_{\pm,\,c})^* =X^{-s,\,-b}_{\mp,\,c},
\end{equation*}
for $0 < c \leq 1$ and $s$, $b \in \R$.
Next we define the angular decomposition of $\R^3$ in frequency. For a dyadic number $A \geq 64$ and an integer 
$j \in [-A, \ A-1]$, we define the sets $\{ {\mathfrak{D}}_j^A \}\subset \R^3$ as follows:
\begin{equation*}
{\mathfrak{D}}_j^A = \left\{ (\tau, |\xi|\cos \theta, |\xi|\sin \theta) 
\in \R \cross \R^2 \ | \ \theta \in  \left[\frac{\pi}{A} \ j, \ \frac{\pi}{A}(j+1) \right] \ \right\}.
\end{equation*}
For any function $u : \R^3 \to \C$,  $\{ {\mathfrak{D}}_j^A \} $ satisfy
\begin{equation*}
\R^3 = \bigcup_{-A \leq j \leq A-1}  {\mathfrak{D}}_j^A, \qquad u = \sum_{j=-A}^{A-1} \chi_{{\mathfrak{D}}_j^A} u \quad a.e.
\end{equation*}
Lastly we introduce the useful two estimates which are called the bilinear Strichartz estimates. The first one holds true regardless of $c$. 
As opposed to that, the second one is given by using the condition $0<c < 1$. The first estimate is obtained by the same argument as in 
the proof of Theorem 2.1 in \cite{Sel2}. We omit the proof.
\begin{prop}[Theorem 2.1. \cite{Sel2}]\label{thm2.1}
Let $0 < c_0, c_1, c_2 \leq 1$. Then we have
\begin{align}
 \| \chi_{K^{\pm_0,c_0}_{N_0,L_0}} & \left( \left( 
\chi_{K^{\pm_1, c_1}_{N_1,L_1}}f \right)
 * \left(\chi_{K^{\pm_2, c_2}_{N_2,L_2}}g \right) \right) 
\|_{L_{\xi,\tau}^2}\notag \\
& \qquad \qquad \lesssim( N_{\textnormal{min}}^{012} L_{\textnormal{min}}^{12}  )^{1/2}  ( N_{\textnormal{min}}^{12} L_{\textnormal{max}}^{12} )^{1/4} \|f\|_{ L_{\xi,\tau}^2} 
\|g \|_{ L_{\xi,\tau}^2},\label{selberg}
\end{align}
regardless of the choice of signs $\pm_j$. Here $*$ denotes the convolution of $\R^3$, 
$N_{\textnormal{min}}^{012} := \textnormal{min}(N_0, N_1, N_2)$, and $N_{\textnormal{min}}^{12}$, 
$L_{\textnormal{min}}^{12}$, $L_{\textnormal{max}}^{12}$ are defined similarly.
\end{prop}
\begin{prop}\label{thm2.2}
Let $0<c<1$. Then we have
\begin{align}
 \| \chi_{K^{\pm_0}_{N_0,L_0}} \left( \left( 
\chi_{K^{\pm_1, c}_{N_1,L_1}}f \right) * \left( \chi_{K^{\pm_2}_{N_2,L_2}}g
\right) \right) \|_{L_{\xi, \tau}^2} \lesssim (N_{\textnormal{min}}^{012} L_1 L_2 )^{1/2} \|f\|_{ L_{\xi,\tau}^2} 
\|g \|_{ L_{\xi,\tau}^2}\label{goal-thm2.2}
\end{align}
regardless of the choice of $\pm_j$.
\end{prop}
\begin{proof}
Let $A =2^{10} (1-c)^{-1/2}$. 
It follows from the finiteness of $A$ and
\begin{equation*}
g = \sum_{j = -A}^{A-1}  \chi_{{\mathfrak{D}}_j^A} g \quad a.e.
\end{equation*}
that we can replace $g$ with $\chi_{{\mathfrak{D}}_j^A} g$ in (\ref{goal-thm2.2}) for fixed $j$. 
After applying rotation in space, we may assume that $j=0$. Also we can assume that there exists $\xi' \in \R^2$ such that 
the support of 
$\chi_{{\mathfrak{D}}_j^A} g$ is contained in the cylinder
\begin{equation*}
C_{N_{\textnormal{min}}^{012}}(\xi') := 
\{ (\tau, \xi) \in \R^3 \ | \ | \xi -\xi' | \leq N_{\textnormal{min}}^{012} \}  .
\end{equation*}
We sketch the validity of the above assumption roughly. See \cite{Tao} for more details. 
If $N_2 \sim N_{\textnormal{min}}^{012}$ the above assumption is harmless obviously. Therefore we may assume that 
$N_0 = N_{\textnormal{min}}^{012} \ll N_2$ or $N_1 = N_{\textnormal{min}}^{012} \ll N_2$. Since both are treated similarly, we 
here consider only the former case. Note that the condition $N_0 \ll N_2$ means $N_2/2 \leq N_1 \leq 2N_2$, otherwise the left-hand side of 
\eqref{goal-thm2.2} vanishes. We can choose the two sets $\{C_{N_{\textnormal{min}}^{012}}(\xi'_k)  \}_k  $ and 
$\{C_{N_{\textnormal{min}}^{012}}(\xi''_\ell)  \}_\ell  $ such that
\begin{align*}
& \# k \sim \left( \frac{N_1}{N_0} \right)^2, \qquad 
\operatorname{supp} \chi_{{\mathfrak{D}}_0^A} g
\subset \bigcup_{k} C_{N_{\textnormal{min}}^{012}}(\xi'_k), \qquad 
|\xi'_k - \xi'_{k'}| \geq N_{\textnormal{min}}^{012} \ \textnormal{for any} \ k, k',\\
& \# \ell \sim \left( \frac{N_1}{N_0} \right)^2, \qquad \, 
\operatorname{supp} f \subset \bigcup_{\ell} C_{N_{\textnormal{min}}^{012}}(\xi''_\ell), \qquad \quad \ \ 
|\xi''_\ell - \xi''_{\ell'}| \geq N_{\textnormal{min}}^{012} \ \textnormal{for any} \ \ell, \ell',
\end{align*}
where $\# k$ and $\# \ell$ denote the numbers of $k$ and $\ell$, respectively. We see that for fixed $k$, 
independently of $N_0$, $N_1$, $N_2$, there is only a 
finite number of $\ell$ which satisfy
\begin{equation*}
\left\|\chi_{K^{\pm_0}_{N_0,L_0}}  \left( \left( \chi_{K^{\pm_1, c}_{N_1,L_1} \cap 
C_{N_{\textnormal{min}}^{012}}(\xi'_k)} f \right) *  \left( 
\chi_{K^{\pm_2}_{N_2,L_2} \cap C_{N_{\textnormal{min}}^{012}}(\xi''_\ell) \cap {\mathfrak{D}}_0^A}  g \right) \right) 
\right\|_{L_{\xi,\tau}^2} >0,
\end{equation*}
and vice versa. This means that $k$ and $\ell$ depend on each other. Once we obtain
\begin{align*}
& \left\|\chi_{K^{\pm_0}_{N_0,L_0}}  \left( \left( \chi_{K^{\pm_1, c}_{N_1,L_1} 
\cap C_{N_{\textnormal{min}}^{012}}(\xi'_k)} f \right) *  \left( 
\chi_{K^{\pm_2}_{N_2,L_2} \cap C_{N_{\textnormal{min}}^{012}}(\xi''_{\ell(k)}) \cap 
{\mathfrak{D}}_0^A}  g \right) \right) \right\|_{L_{\xi,\tau}^2} \\
& \qquad \qquad \qquad \qquad \qquad \qquad \lesssim (N_{\textnormal{min}}^{012} L_1 L_2 )^{1/2} \|f\|_{ L_{\xi,\tau}^2} 
\|g \|_{ L_{\xi,\tau}^2}
\end{align*}
for fixed $k$, from Minkowski inequality and $\ell^2$ almost orthogonality, we confirm
\begin{align*}
& \|\chi_{K^{\pm_0}_{N_0,L_0}}  \left( \left( \chi_{K^{\pm_1, c}_{N_1,L_1}} f \right)
 *  \left( \chi_{K^{\pm_2}_{N_2,L_2}}  g \right) \right) \|_{L_{\xi,\tau}^2} \\
\lesssim & \sum_{k,\ell} \left\|\chi_{K^{\pm_0}_{N_0,L_0}}  \left( \left( \chi_{K^{\pm_1, c}_{N_1,L_1} 
\cap C_{N_{\textnormal{min}}^{012}}(\xi'_k)} f \right) *  \left( 
\chi_{K^{\pm_2}_{N_2,L_2} \cap C_{N_{\textnormal{min}}^{012}}(\xi''_{\ell(k)}) \cap 
{\mathfrak{D}}_0^A}  g \right) \right) \right\|_{L_{\xi,\tau}^2}\\ 
\lesssim & (N_{\textnormal{min}}^{012} L_1 L_2 )^{1/2} \sum_{k, \ell} \|\chi_{C_{N_{\textnormal{min}}^{012}}(\xi'_k)} f\|_{ L_{\xi,\tau}^2} 
\|\chi_{C_{N_{\textnormal{min}}^{012}}(\xi''_{\ell(k)})} g \|_{ L_{\xi,\tau}^2}\\
\lesssim & (N_{\textnormal{min}}^{012} L_1 L_2 )^{1/2} \|f\|_{ L_{\xi,\tau}^2} 
\|g \|_{ L_{\xi,\tau}^2},
\end{align*}
which verify the validity of the assumption. Hereafter, we call the above argument ``$\ell^2$ almost orthogonality''.

We turn to the proof of \eqref{goal-thm2.2}. 
\begin{align*}
& \|\chi_{K^{\pm_0}_{N_0,L_0}}  \left( \left( \chi_{K^{\pm_1, c}_{N_1,L_1}} f\right) *  \left( 
\chi_{K^{\pm_2}_{N_2,L_2} \cap {\mathfrak{D}}_0^A \cap C_{N_{\textnormal{min}}^{012}}(\xi')} 
g \right) \right) \|_{L_{\xi,\tau}^2}\\
\lesssim & \|\chi_{K^{\pm_0}_{N_0,L_0}} \int 
\left(\chi_{K^{\pm_1, c}_{N_1,L_1}}f \right)(\tau- \tau_1, \xi-\xi_1)
\left(\chi_{K^{\pm_2}_{N_2,L_2} \cap {\mathfrak{D}}_0^A \cap C_{N_{\textnormal{min}}^{012}}(\xi') }
 g \right)(\tau_1, \xi_1)d\tau_1d\xi_1 \|_{L_{\xi,\tau}^2}\\
\lesssim & \|\chi_{K^{\pm_0}_{N_0,L_0}} \left( \int |f|^2 (\tau- \tau_1, \xi-\xi_1) 
| g|^2 (\tau_1, \xi_1) d\tau_1d\xi_1 \right)^{1/2} (E(\tau, \xi))^{1/2}  \|_{L_{\xi,\tau}^2}\\
\lesssim & \sup_{(\tau,\xi) \in K^{\pm_0}_{N_0,L_0}} |E(\tau, \xi)|^{1/2} \| |f|^2 * |g|^2 \|_{L_{\xi,\tau}^1}^{1/2}\\
\lesssim & \sup_{(\tau,\xi) \in K^{\pm_0}_{N_0,L_0}} |E(\tau, \xi)|^{1/2} \|f\|_{ L_{\xi,\tau}^2} \|g\|_{ L_{\xi,\tau}^2}
\end{align*}
where
\begin{equation*}
 E(\tau, \xi) 
 := \{(\tau_1, \xi_1) \in C_{N_{\textnormal{min}}^{012}}(\xi') \cap {\mathfrak{D}}_0^A \ | \ \LR{\tau - \tau_1 \pm c |\xi - \xi_1|} \sim L_1, \LR{\tau_1 \pm |\xi_1|} \sim L_2 \}.
\end{equation*}
Thus it suffices to show that 
\begin{equation}
\sup_{\tau,\xi} |E(\tau, \xi) | \lesssim N_{\textnormal{min}}^{012} L_1 L_2.\label{goal-thm2.2-2}
\end{equation}
From $\LR{\tau - \tau_1 \pm c |\xi - \xi_1|} \sim L_1$ and $\LR{\tau_1 \pm |\xi_1|} \sim L_2$, for fixed $\xi_1$, 
\begin{equation}
| \{ \tau_1 \ | \ (\tau_1, \xi_1) \in E(\tau, \xi) \} | \lesssim L_{\textnormal{min}}^{12}.\label{ele1-thm2.2}
\end{equation}
It follows from $(\tau_1, \xi_1) \in {\mathfrak{D}}_0^A$ that
\begin{align}
|\partial_1 (\tau \pm |\xi_1| \pm c|\xi-\xi_1|)| & \geq   \frac{(\xi_1)_1}{|\xi_1|}  - c  \geq \left(  \frac{(\xi_1)_1}{|\xi_1|} \right)^{2} - c \notag \\
& = 1-c - \left(  \frac{(\xi_1)_2}{|\xi_1|} \right)^{2} \geq (1-c)/2,\label{ele2-thm2.2}
\end{align}
where $(\xi_1)_1$ is the first component of $\xi_1$ and $\partial_1$ is the derivative with respect to $(\xi_1)_1.$ 
Combining $|\tau \pm |\xi_1| \pm c|\xi-\xi_1|| \lesssim L_{\textnormal{max}}^{12}$ with (\ref{ele2-thm2.2}), for fixed $(\xi_1)_2$ we have
\begin{equation}
|\{ (\xi_1)_1 \ | \  (\tau_1, \xi_1) \in E(\tau, \xi) \} | \lesssim L_{\textnormal{max}}^{12}\label{ele3-thm2.2}.
\end{equation}
Collecting (\ref{ele1-thm2.2}), (\ref{ele3-thm2.2}) and $\xi_1 \in  C_{N_{\textnormal{min}}^{012}}(\xi') $, we get (\ref{goal-thm2.2-2}).
\end{proof}
\section{Proof of Theorem \ref{nonlinearity-estimate} for $\pm_1 = \pm_2$.}
In (\ref{goal-1-1})-(\ref{goal-2-1}), replacing $u$ and $n$ with its complex conjugates $\bar{u}$ and $\bar{v}$ respectively, 
we easily find that there is no difference between the case $(\pm_0, \pm_1, \pm_2)$ and $(\mp_0, \mp_1, \mp_2)$. 
Here $\mp_j$ denotes a different sign to $\pm_j$. Therefore we assume $\pm_1 =-$ in (\ref{goal-1-1})-(\ref{goal-2-1}) hereafter. 
By the dual argument and Plancherel theorem, we observe that
\begin{align}
(\ref{goal-1-1}) 
\iff \ & \left| \int{ f(\omega_1^{-1} g_1) g_2} dtdx \right| \leq C \|f \|_{X^{s,\,b}_{\pm_0,\,c}} \|g_1 \|_{X^{s,\,b}_{-}}\|g_2 \|_{X^{-s,\,1-b-\e}_{\pm_2}.}\notag \\
\Longleftarrow \ \ \, & \sum_{N_j, L_j (j=0,1,2)} 
 \left|N_1^{-1} \int{ \left( \chi_{K^{\mp_0,c}_{N_0,L_0}} \ha{f} \right) 
\left(
(\chi_{K^{-}_{N_1,L_1}}\ha{g_1}) * (\chi_{K^{\pm_2}_{N_2,L_2}} \ha{g_2})
\right)} 
d\tau d\xi \right|\notag \\
& \qquad \qquad \qquad \qquad \qquad \lesssim \|f \|_{X^{s,\,b}_{\mp_0,\,c}} \|g_1 \|_{X^{s,\,b}_{-}}\|g_2 \|_{X^{-s,\,1-b-\e}_{\pm_2}}.\notag\\
\iff \ & \left( \sum_{N_1} \sum_{1\leq N_0 \lesssim N_1 \sim N_2} + \sum_{N_0} \sum_{1\leq N_1 \lesssim N_0 \sim N_2} + 
\sum_{N_0} \sum_{1\leq N_2 \lesssim N_0 \sim N_1}\right) I_1 \notag \\
& \qquad \qquad \qquad \qquad \qquad \lesssim \|f \|_{\ha{X}^{s,\,b}_{\mp_0,\,c}} \|g_1 \|_{\ha{X}^{s,\,b}_{-}}\|g_2 \|_{\ha{X}^{-s,\,1-b-\e}_{\pm_2}},\label{goal-1-2}
\end{align}
where
\begin{equation*}
I_1 := \sum_{L_j}\left|N_1^{-1} \int{ \left( \chi_{K^{\mp_0,c}_{N_0,L_0}} f \right) 
\left(
(\chi_{K^{-}_{N_1,L_1}} g_1) * (\chi_{K^{\pm_2}_{N_2,L_2}} g_2)
\right)} 
d\tau d\xi \right|.
\end{equation*}
Similarly, (\ref{goal-2-1}) is verified by the following estimate. 
\begin{align}
& \left( \sum_{N_1} \sum_{1\leq N_0 \lesssim N_1 \sim N_2} + \sum_{N_0} \sum_{1\leq N_1 \lesssim N_0 \sim N_2} + 
\sum_{N_0} \sum_{1\leq N_2 \lesssim N_0 \sim N_1}\right) I_2 \notag \\
& \qquad \qquad \qquad \qquad \qquad \lesssim \|f \|_{\ha{X}^{-s,\,1-b-\e}_{\mp_0,\,c}} \|g_1 \|_{\ha{X}^{s,\,b}_{-}}\|g_2 \|_{\ha{X}^{s,\,b}_{\pm_2}}, \label{goal-2-2}
\end{align}
where
\begin{equation*}
I_2 := \sum_{L_j} \left|N_0 N_1^{-1}N_2^{-1}  \int{ \left( \chi_{K^{\mp_0,c}_{N_0,L_0}} f \right) 
\left(
(\chi_{K^{-}_{N_1,L_1}} g_1) * (\chi_{K^{\pm_2}_{N_2,L_2}} g_2)
\right)} 
d\tau d\xi \right|.
\end{equation*}
We now try to establish (\ref{goal-1-2}) and (\ref{goal-2-2}). First we assume that $\pm_2 =-$. In this case, we can obtain  
(\ref{goal-1-2}) and (\ref{goal-2-2}) by using the bilinear estimates Propositions \ref{thm2.1}, \ref{thm2.2} and the following estimate:
\begin{lem}\label{lem modu}
Let $\tau = \tau_1+ \tau_2$, $\xi = \xi_1 + \xi_2$ and $0<c<1$. Then we have
\begin{equation}
\textnormal{max}(\LR{\tau \pm c|\xi|} , \ \LR{\tau_1 - |\xi_1|} , \ \LR{\tau_2 -|\xi_2|} )
 \gtrsim \textnormal{max}(|\xi_1|, \ |\xi_2|).\label{lem modu1}
\end{equation}
\end{lem}
\begin{proof}
\begin{align*}
\textnormal{max}(\LR{\tau \pm c|\xi|} , \ \LR{\tau_1 - |\xi_1|} , \ \LR{\tau_2 -|\xi_2|} )& 
\geq |\tau \pm c |\xi| - (\tau_1 - |\xi_1|) - (\tau_2- |\xi_2|)|\\
& \geq ||\xi_1| + |\xi_2| - c|\xi||\\
& \geq |\xi_1| +|\xi_2| - c(|\xi_1| + |\xi_2|)\\
&  = (1-c)(|\xi_1| + |\xi_2|)
\end{align*}
\end{proof}
\begin{thm}\label{thm2.3}
Let $0<c<1$. For any $s \in (-3/4 , \ 0)$, there exists $b \in (1/2, \ 1)$ such that for $f,g_1, g_2 \in \mathcal{S}(\R \cross \R^2)$, the following estimates hold:
\begin{align}
& \Bigl( \sum_{N_1} \sum_{1\leq N_0 \lesssim N_1 \sim N_2} + \sum_{N_0} \sum_{1\leq N_1 \lesssim N_0 \sim N_2} + 
\sum_{N_0} \sum_{1\leq N_2 \lesssim N_0 \sim N_1}\Bigr) I_1^-  \lesssim \|f \|_{\ha{X}^{s,\,b}_{\pm,\,c}} \|g_1 \|_{\ha{X}^{s,\,b}_{-}}\|g_2 \|_{\ha{X}^{-s,\,1-b-\e}_{-}},
\label{goal1-thm2.3}\\
& \Bigl( \sum_{N_1} \sum_{1\leq N_0 \lesssim N_1 \sim N_2} + \sum_{N_0} \sum_{1\leq N_1 \lesssim N_0 \sim N_2} + 
\sum_{N_0} \sum_{1\leq N_2 \lesssim N_0 \sim N_1}\Bigr) I_2^-  \lesssim \|f \|_{\ha{X}^{-s,\,1-b-\e}_{\pm,\,c}} \|g_1 \|_{\ha{X}^{s,\,b}_{-}}\|g_2 \|_{\ha{X}^{s,\,b}_{-}},\label{goal2-thm2.3}
\end{align}
where
\begin{align*}
I_1^- := & \sum_{L_j}\left|N_1^{-1} \int{ \left( \chi_{K^{\pm,c}_{N_0,L_0}} f \right) 
\left(
(\chi_{K^{-}_{N_1,L_1}} g_1) * (\chi_{K^{-}_{N_2,L_2}} g_2)
\right)} 
d\tau d\xi \right|,\\
I_2^- := & \sum_{L_j} \left|N_0 N_1^{-1}N_2^{-1}  \int{ \left( \chi_{K^{\pm,c}_{N_0,L_0}} f \right) 
\left(
(\chi_{K^{-}_{N_1,L_1}} g_1) * (\chi_{K^{-}_{N_2,L_2}} g_2)
\right)} 
d\tau d\xi \right|.
\end{align*}
\end{thm}
\begin{proof}
For simplicity, we use 
$f^{\pm, c} := \chi_{K^{\pm,c}_{N_0,L_0}} f$, ${g}_k^- := \chi_{K^{-}_{N_k,L_k}}g$ and $(g_k^-)_- ( \, \cdot \, ):= g_k^- (- \, \cdot \, )$ with $k=1,2$. 
Since the proof of (\ref{goal2-thm2.3}) is analogous to that of (\ref{goal1-thm2.3}), we establish only (\ref{goal1-thm2.3}). 
From Lemma \ref{lem modu}, it holds that $L_{\textnormal{max}}^{012} \gtrsim N_{\textnormal{max}}^{12}$. We decompose 
the proof into the three cases:\\
(I) $1\leq N_0 \lesssim N_1 \sim N_2$, \ (II) $1\leq N_1 \lesssim N_0 \sim N_2$, \ (III) $1\leq N_2 \lesssim N_0 \sim N_1.$

First we consider the case (I). 
Considering that $L_{\textnormal{max}}^{012} \gtrsim N_{\textnormal{max}}^{12}$, we subdivide the cases further:\\
(Ia) $N_1 \lesssim L_0$. We deduce from H\"{o}lder inequality and Proposition \ref{thm2.1} that
\begin{align*}
& \sum_{N_1}  \sum_{1\leq N_0 \lesssim N_1 \sim N_2}  \sum_{L_j} 
 \left|N_1^{-1} \int{ f^{\pm, c} (g_1^- * g_2^-)} d\tau d\xi \right|\\
\lesssim & \sum_{N_1} \sum_{1\leq N_0 \lesssim N_1 \sim N_2}  \sum_{L_j} N_1^{-1}  
\| {f^{\pm, c}} \|_{L_{\xi,\tau}^2} \| \chi_{K^{\pm, c}_{N_0,L_0}} (g_1^- * g_2^-) \|_{L_{\xi,\tau}^2}
\\
 \lesssim & \sum_{N_1} \sum_{1\leq N_0 \lesssim N_1 \sim N_2}  \sum_{L_1, L_2} N_1^{-1} N_0^{1/2} 
 N_1^{1/4} L_1^{1/2}L_2^{1/4} N_1^{-b} \| {f^{\pm, c}} \|_{{\ha{X}_{\pm,c}^{0,b}}}\|g_1^- \|_{{L_{\xi,\tau}^2}} \| {g_2^-} \|_{{L_{\xi,\tau}^2}} \\
\lesssim &  \sum_{N_1} \sum_{1\leq N_0 \lesssim N_1 \sim N_2}  N_0^{1/2-s}  N_1^{-3/4-b} N_0^s\| {f^{\pm, c}} \|_{{\ha{X}_{\pm,c}^{0,b}}} N_1^s
\|g_1^- \|_{{\ha{X}_{-}^{0,b}}} N_2^{-s} \| {g_2^-} \|_{{\ha{X}_{-}^{0,1-b-\epsilon}}} \\
\lesssim & \|f \|_{\ha{X}^{s,\,b}_{\pm,\,c}} \|g_1 \|_{\ha{X}^{s,\,b}_{-}}\|g_2 \|_{\ha{X}^{-s,\,1-b-\e}_{-}}.
\end{align*}
(Ib) $N_1 \lesssim L_1$. Similarly, from H\"{o}lder inequality and Proposition \ref{thm2.1} we get
\begin{align*}
& \sum_{N_1}  \sum_{1\leq N_0 \lesssim N_1 \sim N_2}  \sum_{L_j} 
 \left|N_1^{-1} \int{ f^{\pm, c} (g_1^- * g_2^-)} d\tau d\xi \right|\\
\lesssim & \sum_{N_1} \sum_{1\leq N_0 \lesssim N_1 \sim N_2}  \sum_{L_j} N_1^{-1}  
\| \chi_{K^{-}_{N_1,L_1}} ({f^{\pm, c}} * ({g_2^-})_-) \|_{{L_{\xi,\tau}^2}} \|{g_1^-} \|_{{L_{\xi,\tau}^2}} \\
 \lesssim & \sum_{N_1} \sum_{1\leq N_0 \lesssim N_1 \sim N_2}  \sum_{L_0, L_2} N_1^{-1} N_0^{3/4} 
L_0^{1/2}L_2^{1/4} \| {f^{\pm, c}} \|_{{L_{\xi,\tau}^2}} N_1^{-b} \|g_1^- \|_{{\ha{X}_{-}^{0,b}}} \| {g_2^-} \|_{{L_{\xi,\tau}^2}} \\
\lesssim &  \sum_{N_1} \sum_{1\leq N_0 \lesssim N_1 \sim N_2}  N_0^{3/4-s}  N_1^{-1-b} N_0^s\| {f^{\pm, c}} \|_{{\ha{X}_{\pm,c}^{0,b}}} N_1^s
\|g_1^- \|_{{\ha{X}_{-}^{0,b}}} N_2^{-s} \| {g_2^-} \|_{{\ha{X}_{-}^{0,1-b-\epsilon}}} \\
\lesssim & \|f \|_{\ha{X}^{s,\,b}_{\pm,\,c}} \|g_1 \|_{\ha{X}^{s,\,b}_{-}}\|g_2 \|_{\ha{X}^{-s,\,1-b-\e}_{-}}.
\end{align*}
(Ic) $N_1 \lesssim L_2$.
\begin{align*}
&  \sum_{N_1}  \sum_{1\leq N_0 \lesssim N_1 \sim N_2}  \sum_{L_j} 
 \left|N_1^{-1} \int{ f^{\pm, c} (g_1^- * g_2^-)} d\tau d\xi \right|\\
\lesssim & \sum_{N_1} \sum_{1\leq N_0 \lesssim N_1 \sim N_2}  \sum_{L_j} N_1^{-1}  
\| \chi_{K^{-}_{N_2,L_2}} ({f^{\pm, c}} * ({g_1^-})_-) \|_{{L_{\xi,\tau}^2}} \|{g_2^-} \|_{{L_{\xi,\tau}^2}}\\
 \lesssim & \sum_{N_1} \sum_{1\leq N_0 \lesssim N_1 \sim N_2}  \sum_{L_0, L_1} N_1^{-1} N_0^{3/4} 
L_0^{1/2}L_1^{1/4} \| {f^{\pm, c}} \|_{{L_{\xi,\tau}^2}} \|g_1^- \|_{{L_{\xi,\tau}^2}}  N_1^{-1+b+\e} \| {g_2^-} \|_{{\ha{X}_{-}^{0,1-b-\epsilon}}} \\
\lesssim &  \sum_{N_1} \sum_{1\leq N_0 \lesssim N_1 \sim N_2}  N_0^{3/4-s}  N_1^{-2+b+\e} N_0^s\| {f^{\pm, c}} \|_{{\ha{X}_{\pm,c}^{0,b}}} N_1^s
\|g_1^- \|_{{\ha{X}_{-}^{0,b}}} N_2^{-s} \| {g_2^-} \|_{{\ha{X}_{-}^{0,1-b-\epsilon}}} \\
\lesssim & \|f \|_{\ha{X}^{s,\,b}_{\pm,\,c}} \|g_1 \|_{\ha{X}^{s,\,b}_{-}}\|g_2 \|_{\ha{X}^{-s,\,1-b-\e}_{-}}.
\end{align*}
For the case (II), we can show (\ref{goal1-thm2.3}) in the same manner as above. We omit the proof.
Lastly, we consider the case (III).\\
(IIIa) $N_0 \lesssim L_0$. We deduce from H\"{o}lder inequality and Proposition \ref{thm2.1} that
\begin{align*}
& \sum_{N_0}  \sum_{1\leq N_2 \lesssim N_0 \sim N_1}  \sum_{L_j} 
 \left|N_1^{-1} \int{ f^{\pm, c} (g_1^- * g_2^-)} d\tau d\xi \right|\\
\lesssim & \sum_{N_0}  \sum_{1\leq N_2 \lesssim N_0 \sim N_1}  \sum_{L_j} N_0^{-1}  
\| {f^{\pm, c}} \|_{L_{\xi,\tau}^2} \| \chi_{K^{\pm}_{N_0,L_0}} (g_1^- * g_2^-) \|_{L_{\xi,\tau}^2}
\\
 \lesssim & \sum_{N_0}  \sum_{1\leq N_2 \lesssim N_0 \sim N_1}  \sum_{L_1, L_2}  N_0^{-1} N_2^{3/4} 
L_1^{1/2}L_2^{1/4} N_0^{-b} \| {f^{\pm, c}} \|_{{\ha{X}_{\pm,c}^{0,b}}}\|g_1^- \|_{{L_{\xi,\tau}^2}} \| {g_2^-} \|_{{L_{\xi,\tau}^2}} \\
\lesssim &  \sum_{N_0}  \sum_{1\leq N_2 \lesssim N_0 \sim N_1} N_0^{-1-b-2s}  N_2^{3/4 + s} N_0^s\| {f^{\pm, c}} \|_{{\ha{X}_{\pm,c}^{0,b}}} N_1^s
\|g_1^- \|_{{\ha{X}_{-}^{0,b}}} N_2^{-s} \| {g_2^-} \|_{{\ha{X}_{-}^{0,1-b-\epsilon}}} \\
\lesssim & \|f \|_{\ha{X}^{s,\,b}_{\pm,\,c}} \|g_1 \|_{\ha{X}^{s,\,b}_{-}}\|g_2 \|_{\ha{X}^{-s,\,1-b-\e}_{-}}.
\end{align*}
(IIIb) $N_0 \lesssim L_1$. Similarly, 
\begin{align*}
& \sum_{N_0}  \sum_{1\leq N_2 \lesssim N_0 \sim N_1}  \sum_{L_j} 
 \left|N_1^{-1} \int{ f^{\pm, c} (g_1^- * g_2^-)} d\tau d\xi \right|\\
\lesssim & \sum_{N_0}  \sum_{1\leq N_2 \lesssim N_0 \sim N_1}  \sum_{L_j} N_0^{-1}  
\| \chi_{K^{-}_{N_1,L_1}} ({f^{\pm, c}} * ({g_2^-})_-) \|_{{L_{\xi,\tau}^2}} \|{g_1^-} \|_{{L_{\xi,\tau}^2}}\\
 \lesssim & \sum_{N_0}  \sum_{1\leq N_2 \lesssim N_0 \sim N_1}  \sum_{L_0, L_2}  N_0^{-1} N_2^{3/4} 
L_0^{1/2}L_2^{1/4} \| {f^{\pm, c}} \|_{{L_{\xi,\tau}^2}} N_0^{-b} \|g_1^- \|_{{\ha{X}_{-}^{0,b}}} \| {g_2^-} \|_{{L_{\xi,\tau}^2}} \\
\lesssim &  \sum_{N_0}  \sum_{1\leq N_2 \lesssim N_0 \sim N_1} N_0^{-1-b-2s}  N_2^{3/4 + s} N_0^s\| {f^{\pm, c}} \|_{{\ha{X}_{\pm,c}^{0,b}}} N_1^s
\|g_1^- \|_{{\ha{X}_{-}^{0,b}}} N_2^{-s} \| {g_2^-} \|_{{\ha{X}_{-}^{0,1-b-\epsilon}}} \\
\lesssim & \|f \|_{\ha{X}^{s,\,b}_{\pm,\,c}} \|g_1 \|_{\ha{X}^{s,\,b}_{-}}\|g_2 \|_{\ha{X}^{-s,\,1-b-\e}_{-}}.
\end{align*}
(IIIc) $N_0 \lesssim L_2$. In this case, we need to utilize Proposition \ref{thm2.2} instead of Proposition \ref{thm2.1}.
\begin{align*}
& \sum_{N_0}  \sum_{1\leq N_2 \lesssim N_0 \sim N_1}  \sum_{L_j} 
 \left|N_1^{-1} \int{ f^{\pm, c} (g_1^- * g_2^-)} d\tau d\xi \right|\\
\lesssim & \sum_{N_0}  \sum_{1\leq N_2 \lesssim N_0 \sim N_1}  \sum_{L_j} N_0^{-1}  
\| P_{K^{-}_{N_2,L_2}} ({f^{\pm, c}} {g_1^-}) \|_{{L_{\xi,\tau}^2}} \|{g_2^-} \|_{{L_{\xi,\tau}^2}}\\
 \lesssim & \sum_{N_0}  \sum_{1\leq N_2 \lesssim N_0 \sim N_1}  \sum_{L_0, L_1}  N_0^{-1} N_2^{1/2} 
L_0^{1/2}L_1^{1/2} \| {f^{\pm, c}} \|_{{L_{\xi,\tau}^2}} \|g_1^- \|_{{L_{\xi,\tau}^2}} N_0^{-1+b+\e} \| {g_2^-} \|_{{\ha{X}_{-}^{0,1-b-\epsilon}}}\\
\lesssim &  \sum_{N_0}  \sum_{1\leq N_2 \lesssim N_0 \sim N_1} N_0^{-2+b-2s+\e}  N_2^{1/2 + s} N_0^s\| {f^{\pm, c}} \|_{{\ha{X}_{\pm,c}^{0,b}}} N_1^s
\|g_1^- \|_{{\ha{X}_{-}^{0,b}}} N_2^{-s} \| {g_2^-} \|_{{\ha{X}_{-}^{0,1-b-\epsilon}}} \\
\lesssim & \|f \|_{\ha{X}^{s,\,b}_{\pm,\,c}} \|g_1 \|_{\ha{X}^{s,\,b}_{-}}\|g_2 \|_{\ha{X}^{-s,\,1-b-\e}_{-}}.
\end{align*}
\end{proof}
\section{Proof of Theorem \ref{nonlinearity-estimate} for $\pm_1 \not= \pm_2$.}
In this section, we establish (\ref{goal-1-2}) and (\ref{goal-2-2}) with $\pm_2 = +$. 
Note that if one of the inequalities $|\xi_2| \leq \frac{1-c}{2 (1+c)} |\xi_1|$ and $|\xi_1| \leq \frac{1-c}{2 (1+c)} |\xi_2|$ 
holds, then we observe that for $\tau = \tau_1+ \tau_2$, $\xi = \xi_1 + \xi_2$,
\begin{align*}
\textnormal{max}(\LR{\tau \pm_0 c|\xi|} , \ \LR{\tau_1 - |\xi_1|} , \ \LR{\tau_2 + |\xi_2|} )
& \geq |\pm c |\xi| -  |\xi_1| + |\xi_2||\\
& \geq ||\xi_1| - |\xi_2||- c|\xi_1|-c |\xi_2|\\
& \geq \frac{1-c}{2} \textnormal{max}(|\xi_1|, \ |\xi_2|)
\end{align*}
and we can verify (\ref{goal-1-2}) and (\ref{goal-2-2}) by the same proof as in the case $\pm_2 = -$. To avoid redundancy, we omit 
the proof.
\begin{prop}\label{prop-hojo}
Let $0<c<1$. 
For any $s \in (-3/4 , \, 0)$, there exists $b \in (1/2, \, 1)$ such that for $f,g_1, g_2 \in \mathcal{S}(\R \cross \R^2)$, the following estimates hold:
\begin{align*}
& \Bigl( \sum_{N_0} \sum_{1\leq N_1 \ll N_0 \sim N_2} + 
\sum_{N_0} \sum_{1\leq N_2 \ll N_0 \sim N_1}\Bigr) I_1^+ \lesssim \|f \|_{\ha{X}^{s,\,b}_{\pm,\,c}} \|g_1 \|_{\ha{X}^{s,\,b}_{-}}\|g_2 \|_{\ha{X}^{-s,\,1-b-\e}_{+}},
\\
& \Bigl( \sum_{N_0} \sum_{1\leq N_1 \ll N_0 \sim N_2} + 
\sum_{N_0} \sum_{1\leq N_2 \ll N_0 \sim N_1}\Bigr) I_2^+  \lesssim \|f \|_{\ha{X}^{-s,\,1-b-\e}_{\pm,\,c}} \|g_1 \|_{\ha{X}^{s,\,b}_{-}}\|g_2 \|_{\ha{X}^{s,\,b}_{+}},
\end{align*}
where
\begin{align*}
I_1^+ := & \sum_{L_j} \left|N_1^{-1} \int{ \left( \chi_{K^{\pm,c}_{N_0,L_0}} f \right) 
\left(
(\chi_{K^{-}_{N_1,L_1}} g_1) * (\chi_{K^{+}_{N_2,L_2}} g_2)
\right)} 
d\tau d\xi \right|,\\
I_2^+ := & \sum_{L_j} \left|N_0 N_1^{-1}N_2^{-1}  \int{ \left( \chi_{K^{\pm,c}_{N_0,L_0}} f \right) 
\left(
(\chi_{K^{-}_{N_1,L_1}} g_1) * (\chi_{K^{+}_{N_2,L_2}} g_2)
\right)} 
d\tau d\xi \right|.
\end{align*}
\end{prop}
Thanks to Proposition \ref{prop-hojo}, 
we can assume that $1\leq N_0 \lesssim N_1 \sim N_2$. 
In this case, we no longer make use of the useful 
estimate such as (\ref{lem modu1}) and, as we mentioned in the introduction, it appears that the bilinear  estimates 
Propositions \ref{thm2.1}, \ref{thm2.2} are not 
enough to show (\ref{goal-1-2}) and (\ref{goal-2-2}). Thus we employ a new estimate developed by Bejenaru-Herr-Tataru 
\cite{BHT} and applied to Zakharov system in \cite{BHHT}, \cite{BH}. 
To describe it precisely, we introduce the decomposition of $\R^3 \cross \R^3$ which was exploited in \cite{BHHT}. 
For dyadic numbers $M_0$, $M_1$, to be chosen later, we decompose $\R^3 \cross \R^3$ into the sets $\{ {\mathfrak{D}}_j^A\}$.
\begin{align*}
\R^3 \cross \R^3 = & \left\{ \angle(\xi_1, \ \xi_2) \leq \frac{16}{M_0}\pi \right\} \cup 
\bigcup_{64 \leq A \leq M_0} \left\{ \frac{16}{A} \pi \leq \angle (\xi_1, \ \xi_2) \leq \frac{32}{A} \pi \right\}\\
 \cup & \left\{ \pi - \frac{16}{M_1}\pi \leq \angle(\xi_1, \ \xi_2) \right\} \cup 
\bigcup_{64 \leq A \leq M_1} \left\{ \pi - \frac{32}{A} \pi \leq \angle (\xi_1, \ \xi_2) \leq \pi - \frac{16}{A} \pi \right\}\\
= & \bigcup_{\tiny{\substack{-M_0 \leq j_1,j_2 \leq M_0 -1\\|j_1 - j_2|\leq 16}}} 
{\mathfrak{D}}_{j_1}^{M_0} \cross {\mathfrak{D}}_{j_2}^{M_0} \cup \bigcup_{64 \leq A \leq M_0} \ \bigcup_{\tiny{\substack{-A \leq j_1,j_2 \leq A -1\\ 16 \leq |j_1 - j_2|\leq 32}}} 
{\mathfrak{D}}_{j_1}^A \cross {\mathfrak{D}}_{j_2}^A \\
\cup & \bigcup_{\tiny{\substack{-M_1 \leq j_1,j_2 \leq M_1 -1\\|j_1 - j_2 \pm M_1|\leq 16}}} 
{\mathfrak{D}}_{j_1}^{M_1} \cross {\mathfrak{D}}_{j_2}^{M_1} \cup \bigcup_{64 \leq A \leq M_1} \ 
\bigcup_{\tiny{\substack{-A \leq j_1,j_2 \leq A +1\\ 16 \leq |j_1 - j_2 \pm A |\leq 32}}} 
{\mathfrak{D}}_{j_1}^A \cross {\mathfrak{D}}_{j_2}^A,
\end{align*}
where $\angle (\xi_1, \ \xi_2) \in [0, \ \pi]$ is the smaller angle between $\xi_1$ and $\xi_2$. 

First we assume that $\pi/2 \leq \angle (\xi_1, \xi_2) \leq \pi$. 
We find that if $\angle (\xi_1, \ \xi_2)$ is sufficiently close to $\pi$ , then the following helpful inequality holds true.
\begin{lem}\label{lem-modu1}
Let $\tau = \tau_1 + \tau_2$, $\xi = \xi_1 +\xi_2$, $0<c<1$ and  $M_1$ be the minimal dyadic number which satisfies 
\begin{equation}
M_1 \geq 2^{7} (1-c)^{-\frac{1}{2}} \frac{(|\xi_1||\xi_2|)^{\frac{1}{2}}}{|\xi|},\label{assumption-lem2.5}
\end{equation}
then for any $(\tau_1, \xi_1) \in {\mathfrak{D}}_{j_1}^{M_1}$, $(\tau_2, \xi_2) \in {\mathfrak{D}}_{j_2}^{M_1}$ where $|j_1 - j_2 \pm M_1|\leq 16$,  the following inequality holds;
\begin{equation*}
\textnormal{max}(\LR{\tau \pm c|\xi|} , \ \LR{\tau_1 - |\xi_1|} , \ \LR{\tau_2 + |\xi_2|} ) \gtrsim |\xi|
\end{equation*}
\end{lem}
\begin{proof}
After rotation, we may assume $\xi_1 = (|\xi_1|,  0)$, and then $|j_2 \pm M_1|\leq 16$. It follows from the inequality
\begin{equation*}
\textnormal{max}(\LR{\tau \pm c|\xi|} , \ \LR{\tau_1 - |\xi_1|} , \ \LR{\tau_2 + |\xi_2|} )\geq ||\xi_1|-|\xi_2|| - c |\xi|,
\end{equation*}
it suffices to show $||\xi_1|-|\xi_2|| > \sqrt{\frac{1+c}{2}} |\xi|$. Indeed, 
\begin{equation*}
\sqrt{\frac{1+c}{2}} - c > \frac{1}{4} (1-c)(1+2c)>\frac{1-c}{4}.
\end{equation*}
From $|j_2 \pm M_1|\leq 16$, we obtain 
\begin{align*}
|\xi|^2 & = (|\xi_1|+|\xi_2| \cos (\angle (\xi_1, \ \xi_2) ))^2 + (N_2 \sin (\angle (\xi_1, \ \xi_2) ))^2 \\
& < (|\xi_1| -|\xi_2|)^2 + 2 |\xi_1||\xi_2|(1 + \cos (\angle (\xi_1, \ \xi_2) ))\\
& < (|\xi_1| -|\xi_2|)^2 + 4 |\xi_1| |\xi_2| (\angle (\xi_1, \ \xi_2))^2\\
& < (|\xi_1| -|\xi_2|)^2 + \frac{1-c}{2} |\xi|^2,
\end{align*}
which gives
\begin{equation*}
\frac{1+c}{2} |\xi|^2 < (|\xi_1| -|\xi_2|)^2.
\end{equation*}
This completes the proof.
\end{proof}
Next we consider the case $64 \leq A \leq M_1$ and $16 \leq |j_1 - j_2 \pm A |\leq 32$.
\begin{prop}\label{thm2.6}
Let $\tau = \tau_1+\tau_2$, $\xi=\xi_1+\xi_2$, $0<c<1$ and $f$, $g_1$, $g_2 \in L^2$ be satisfy
\begin{equation*}
\operatorname{supp} f \subset K^{\pm,c}_{N_0,L_0}, \quad \operatorname{supp} g_1 \subset {\mathfrak{D}}_{j_1}^A \cap K^{-}_{N_1,L_1}, \quad 
 \operatorname{supp} g_2 \subset {\mathfrak{D}}_{j_2}^A \cap K^{+}_{N_2,L_2},
\end{equation*}
and $64 \leq N_0 \lesssim N_1 \sim N_2$, \ $64 \leq A \leq M_1$, \ $16 \leq |j_1 - j_2 \pm A |\leq 32$. 
Then the following estimate holds;
\begin{equation*}
\left| \int f(\tau, \xi) g_1(\tau_1, \xi_1) g_2 (\tau_2, \xi_2) d\tau_1 d\tau_2 d\xi_1 d \xi_2 \right| \lesssim A^{\frac{7}{8}} (L_0 L_1 L_2)^\frac{1}{2}
\|f\|_{{L_{\xi,\tau}^2}} \|g_1 \|_{{L_{\xi,\tau}^2}} \|g_2 \|_{{L_{\xi,\tau}^2}}.
\end{equation*}
\end{prop}
For the proof of the above proposition, we introduce the important estimate. See \cite{BH} for more general case.
\begin{prop}[\cite{BHT} Corollary 1.5] \label{prop2.7}
Assume that the surface $\tilde{S_i}$ $(i=1,2,3)$ 
is an open and bounded subset of $\tilde{S_i^*}$ which 
satisfies the following conditions \textnormal{(Assumption 1.1 in \cite{BHT})}.

\textnormal{(i)} $\tilde{S_i^*}$ is defined as
\begin{equation*}
\tilde{S_i^*} = \{ \tilde{\sigma_i} \in U_i \ | \ \Phi_i(\tilde{\sigma_i}) = 0 , \nabla \Phi_i \not= 0, \Phi_i \in C^{1,1} (U_i) \}, 
\end{equation*}
for a convex $U_i \subset \R^3$ such that \textnormal{dist}$(\tilde{S_i}, U_i^c) \geq$ \textnormal{diam}$(\tilde{S_i})$;

\textnormal{(ii)} the unit normal vector field $\tilde{\mathfrak{n}_i}$ on $\tilde{S_i^*}$ satisfies the H\"{o}lder condition
\begin{equation*}
\sup_{\tilde{\sigma}, \tilde{\sigma}' \in \tilde{S_i^*}} \frac{|\tilde{\mathfrak{n}_i}(\tilde{\sigma}) - 
\tilde{\mathfrak{n}_i}(\tilde{\sigma}')|}{|\tilde{\sigma} - \tilde{\sigma}'|}
+ \frac{|\tilde{\mathfrak{n}_i}(\tilde{\sigma}) (\tilde{\sigma} - \tilde{\sigma}')|}{|\tilde{\sigma} - \tilde{\sigma}'|^2} \lesssim 1;
\end{equation*}

\textnormal{(iii)} the matrix $\tilde{N}(\tilde{\sigma_1}, \tilde{\sigma_2}, \tilde{\sigma_3}) = (\tilde{\mathfrak{n}_1} 
(\tilde{\sigma_1}), \tilde{\mathfrak{n}_2}(\tilde{\sigma_2}), \tilde{\mathfrak{n}_3}(\tilde{\sigma_3}))$ 
satisfies the transversality condition
\begin{equation*}
\frac{1}{2} \leq \textnormal{det} \tilde{N}(\tilde{\sigma_1}, \tilde{\sigma_2}, \tilde{\sigma_3})  \leq 1
\end{equation*}
for all $(\tilde{\sigma_1}, \tilde{\sigma_2}, \tilde{\sigma_3}) \in \tilde{S_1^*} \cross \tilde{S_2^*} \cross \tilde{S_3^*}$.

We also assume \textnormal{diam}$(\tilde{S_i}) \leq 1$. Let $T : \R^3 \to \R^3$ be an invertible, linear map and $S_i = T\tilde{S_i} $. 
Then for functions $f \in L^2 (S_1)$ and $g \in L^2 (S_2)$, the restriction of the convolution $f*g$ to 
$S_3$ is a well-defined $L^2(S_3)$-function which satisfies 
\begin{equation*}
\| f *g \|_{L^2(S_3)} \lesssim \frac{1}{\sqrt{d}} \| f \|_{L^2(S_1)} \| g\|_{L^2(S_2)},
\end{equation*}
where 
\begin{equation*}
d = \inf_{\sigma_i \in S_i} |\textnormal{det} N (\sigma_1, \sigma_2, \sigma_3)|
\end{equation*}
and $N (\sigma_1, \sigma_2, \sigma_3)$ is the matrix of the unit normals to $S_i$ at 
$ (\sigma_1, \sigma_2, \sigma_3)$.
\end{prop}
\begin{rem}
As was mentioned in \cite{BHT}, the condition of $S_i^*$ \textnormal{(i)} is used only to ensure the existence of a global representation of 
$S_i$ as a graph. In the proof of Proposition \ref{thm2.6}, the implicit function theorem and the other conditions may show the existence 
of such a graph. Thus we will not treat the condition \textnormal{(i)} in the proof of Proposition \ref{thm2.6}.
\end{rem}
By utilizing Proposition \ref{prop2.7}, we verify Proposition \ref{thm2.6}. 
\begin{proof}[Proof of Proposition \ref{thm2.6}]
Let $\theta_0^\pm \in (0,\pi)$ be defined as $\cos \theta_0^\pm = \pm c$. We divide the proof into the following two cases:
\begin{align*}
\textnormal{(I)} \quad & |\angle(\xi, \xi_1) - \theta_0^+|>2^{10}(1-c)^{-1}A^{-3/4} \ \textnormal{and} \ 
 |\angle(\xi, \xi_1) - \theta_0^-|>2^{10}(1-c)^{-1}A^{-3/4},
\\
\textnormal{(I \hspace{-0.15cm}I)} \quad & |\angle(\xi, \xi_1) - \theta_0^+|\leq 2^{10}(1-c)^{-1}A^{-3/4} 
\  \ \textnormal{or} \ \ |\angle(\xi, \xi_1) - \theta_0^-|\leq 2^{10}(1-c)^{-1}A^{-3/4},
\end{align*}
where $\angle(\xi, \xi_1) \in (0, \pi)$ is the smaller angle between $\xi$ and $\xi_1$. 
We here assume that $A > 2^{20} (1-c)^{-2}$. If $A \leq 2^{20} (1-c)^{-2}$, the proposition is verified by the 
almost same proof as that for the case \textnormal{(I \hspace{-0.15cm}I)} below.

We first consider the case (I). The proof is very similar to that for $\pm_1 = \pm_2$. We utilize the following two estimates.
\begin{lem}\label{lem-a-modu}
Let $\tau = \tau_1 + \tau_2$, $\xi = \xi_1 +\xi_2$, $0<c<1$, $2^{20} (1-c)^{-2} < A \leq M_1$ and 
$\angle(\xi, \xi_1) $ satisfies \textnormal{(I)}. Then the following inequality holds:
\begin{equation*}
\textnormal{max}(\LR{\tau \pm c|\xi|} , \ \LR{\tau_1 - |\xi_1|} , \ \LR{\tau_2 + |\xi_2|} ) \gtrsim A^{-3/4}|\xi|
\end{equation*}
\end{lem}
\begin{proof}
After rotation, we may assume that $\xi_1 = (|\xi_1|, 0)$ and $\xi = (|\xi| \cos \theta, |\xi| \sin \theta)$ with 
$\theta := \angle(\xi, \xi_1) \in (0, \pi)$. By the simple calculation, we have
\begin{align*}
& \textnormal{max}(  \LR{\tau \pm c|\xi|} , \ \LR{\tau_1 - |\xi_1|} ,  \ \LR{\tau_2 + |\xi_2|} ) \\
\geq & |\pm c |\xi| +  |\xi_1| - |\xi_2||\\
= & \left| \pm c |\xi| +  |\xi_1| - \sqrt{|\xi_1|^2 - 2 |\xi||\xi_1|\cos \theta + |\xi|^2} \right|\\
\geq & \left| \pm c |\xi| + |\xi| \cos \theta \right| - 
\left| \frac{2 |\xi| |\xi_1| \cos \theta - |\xi|^2}{ |\xi_1| + \sqrt{|\xi_1|^2 - 2 |\xi||\xi_1|\cos \theta + |\xi|^2}}
- |\xi|\cos \theta \right|\\
=: & \  K_1 - K_2.
\end{align*}
From $\theta_0^\pm, \theta \in (0, \pi)$ and \textnormal{(I)}, we get
\begin{align*}
K_1  = |\xi||\cos \theta_0^\mp - \cos \theta| 
& \geq |\xi| \ \frac{\sqrt{1-c}}{4} \ |\theta_0^\mp - \theta|\\
& \geq 2^8 (1-c)^{-\frac{1}{2}}|\xi|A^{-\frac{3}{4}}.
\end{align*}
From $2^{20} (1-c)^{-2} < A \leq M_1$, we have
\begin{align*}
K_2 & = 
\left| \frac{2 |\xi| |\xi_1| \cos \theta - |\xi|^2}{ |\xi_1| + \sqrt{|\xi_1|^2 - 2 |\xi||\xi_1|\cos \theta + |\xi|^2}}
- |\xi|\cos \theta \right|\\
& \leq  
\left| \frac{2 |\xi| |\xi_1| \cos \theta}{ |\xi_1| + \sqrt{|\xi_1|^2 - 2 |\xi||\xi_1|\cos \theta + |\xi|^2}}
- \frac{2|\xi| |\xi_1| \cos \theta}{2 |\xi_1|} \right| + \frac{|\xi|^2}{|\xi_1|}\\
& \leq 
\frac{|\xi| |\xi_1| | \cos \theta|}{|\xi_1| ( \sqrt{|\xi_1|^2 - 2 |\xi||\xi_1|\cos \theta + |\xi|^2})} 
\left| |\xi_1| - \sqrt{|\xi_1|^2 - 2 |\xi||\xi_1|\cos \theta + |\xi|^2} \right|  + \frac{|\xi|^2}{|\xi_1|}\\
& \leq 
4 \frac{|\xi|^2}{|\xi_1|} \leq 2^{10} (1-c)^{-1} |\xi| A^{-1} \\
& \leq 2^5 (1-c)^{-\frac{1}{2}}|\xi|A^{-\frac{3}{4}}.
\end{align*}
From above, we have
\begin{equation*}
K_1-K_2 \gtrsim |\xi| A^{-\frac{3}{4}}.
\end{equation*}
This completes the proof.
\end{proof}
\begin{lem}\label{lem-a-bilinear}
Let $g_1$, $g_2 \in L^2$ be satisfy
\begin{equation*}
\operatorname{supp} g_1 \subset {\mathfrak{D}}_{j_1}^A \cap K^{-}_{N_1,L_1}, \quad \operatorname{supp} g_2 \subset {\mathfrak{D}}_{j_2}^A \cap K^{+}_{N_2,L_2},
\end{equation*}
and $64 \leq N_0 \lesssim N_1 \sim N_2$, \ $64 \leq A \leq M_1$, \ $16 \leq |j_1 - j_2 \pm A |\leq 32$. 
Then the following estimate holds;
\begin{align*}
& \| \chi_{K^{\pm_0}_{N_0,L_0}}  \left( g_1 * g_2  \right) \|_{L_{\xi,\tau}^2} 
\lesssim (A N_0 L_1 L_2 )^{\frac{1}{2}} \|g_1\|_{ L_{\xi, \tau}^2} 
\|g_2 \|_{ L_{\xi, \tau}^2}
\end{align*}
\end{lem}
\begin{proof}
By the same way as in the proof of Proposition \ref{thm2.2}, we observe that the desired estimate 
is proved by
\begin{equation}
\sup_{\tau,\xi} |E(\tau, \xi) | \lesssim A N_0 L_1 L_2\label{goal001-lem-a-bilinear}
\end{equation}
where
\begin{equation*} E(\tau, \xi) 
 := \left\{ (\tau_1, \xi_1)  \in  {\mathfrak{D}}_{0}^{A} \cap C_{N_0}(\xi') \ \left| \ 
\begin{aligned} & \LR{\tau - \tau_1 - |\xi - \xi_1|} \sim L_1, \ \LR{\tau_1 + |\xi_1|} \sim L_2,\\ 
& (\tau-\tau_1, \xi-\xi_1) \in {\mathfrak{D}}_{j_2}^{A}.
 \end{aligned} \right.
\right\}
\end{equation*}
with $16 \leq |j_2 \pm A| \leq 32$ and fixed $\xi' \in \R^2$. 
From $\LR{\tau - \tau_1 - |\xi - \xi_1|} \sim L_1$ and $\LR{\tau_1 + |\xi_1|} \sim L_2$, for fixed $\xi_1$, 
\begin{equation}
| \{ \tau_1 \ | \ (\tau_1, \xi_1) \in E(\tau, \xi) \} | \lesssim L_{\textnormal{min}}^{12}.\label{ele1-lem-a-bilinear}
\end{equation}
It follows from $(\tau_1, \xi_1) \in {\mathfrak{D}}_0^A$ and 
$(\tau-\tau_1, \xi-\xi_1) \in {\mathfrak{D}}_{j_2}^{A}$ that
\begin{align}
|\partial_2 (\tau - |\xi_1| + |\xi-\xi_1|)| & \geq   \left| \frac{(\xi_1)_2}{|\xi_1|}  + 
\frac{(\xi - \xi_1)_2}{|\xi - \xi_1|} \right| \notag \\
& \gtrsim A^{-1}. \label{ele2-lem-a-bilinear}
\end{align}
Combining $|\tau - |\xi_1| + |\xi-\xi_1|| \lesssim L_{\textnormal{max}}^{12}$ with (\ref{ele2-lem-a-bilinear}), 
for fixed $(\xi_1)_1$ we have
\begin{equation}
|\{ (\xi_1)_2 \ | \  (\tau_1, \xi_1) \in E(\tau, \xi) \} | 
\lesssim A L_{\textnormal{max}}^{12}\label{ele3-lem-a-bilinear}.
\end{equation}
Collecting (\ref{ele1-lem-a-bilinear}), (\ref{ele3-lem-a-bilinear}) and $\xi_1 \in  C_{N_0}(\xi') $, we get (\ref{goal001-lem-a-bilinear}).
\end{proof}
We now prove Proposition \ref{thm2.6} for the case (I). 
From Lemma \ref{lem-a-modu}, it holds that $L_{\textnormal{max}}^{012} \gtrsim A^{-\frac{3}{4}} 
N_0$. We decompose the proof into the three cases:\\
\qquad (Ia) $A^{-\frac{3}{4}} N_0 \lesssim L_0$, \qquad (Ib) $A^{-\frac{3}{4}} N_0 \lesssim L_1$, \qquad (Ic) $A^{-\frac{3}{4}} N_0 \lesssim L_2.$\\
(Ia) From H\"{o}lder inequality and Lemma \ref{lem-a-bilinear}, we have
\begin{align*}
\left| \int f(\tau, \xi) g_1(\tau_1, \xi_1) g_2 (\tau_2, \xi_2) d\tau_1 d\tau_2 d\xi_1 d \xi_2 \right| & 
\lesssim \| f\|_{{L_{\xi,\tau}^2}} \|g_1 * g_2 \|_{{L_{\xi,\tau}^2}} \notag \\
& \lesssim  A^{\frac{7}{8}} (L_0 L_1 L_2)^\frac{1}{2}
\|f\|_{{L_{\xi,\tau}^2}} \|g_1 \|_{{L_{\xi,\tau}^2}} \|g_2 \|_{{L_{\xi,\tau}^2}}.
\end{align*}
(Ib) From H\"{o}lder inequality and Lemma \ref{thm2.2}, we have
\begin{align*}
\left| \int f(\tau, \xi) g_1(\tau_1, \xi_1) g_2 (\tau_2, \xi_2) d\tau_1 d\tau_2 d\xi_1 d \xi_2 \right| & 
\lesssim \| g_1\|_{{L_{\xi,\tau}^2}} \|f * g_{2,-} \|_{{L_{\xi,\tau}^2}} \notag \\
& \lesssim  A^{\frac{3}{8}} (L_0 L_1 L_2)^\frac{1}{2}
\|f\|_{{L_{\xi,\tau}^2}} \|g_1 \|_{{L_{\xi,\tau}^2}} \|g_2 \|_{{L_{\xi,\tau}^2}}.
\end{align*}
(Ic) From H\"{o}lder inequality and Lemma \ref{thm2.2}, we have
\begin{align*}
\left| \int f(\tau, \xi) g_1(\tau_1, \xi_1) g_2 (\tau_2, \xi_2) d\tau_1 d\tau_2 d\xi_1 d \xi_2 \right| & 
\lesssim \| g_2\|_{{L_{\xi,\tau}^2}} \|f * g_{1,-} \|_{{L_{\xi,\tau}^2}} \notag \\
& \lesssim  A^{\frac{3}{8}} (L_0 L_1 L_2)^\frac{1}{2}
\|f\|_{{L_{\xi,\tau}^2}} \|g_1 \|_{{L_{\xi,\tau}^2}} \|g_2 \|_{{L_{\xi,\tau}^2}}.
\end{align*}
Here $g_{j,-}$ is defined as $g_{j,-} (\cdot) = g_{j}(- \cdot).$ 

We next consider the case \textnormal{(I \hspace{-0.15cm}I)}. We apply the same strategy as 
that of the proof of 
Proposition 4.4 in \cite{BHHT}. 
Applying the transformation $\tau_1 = |\xi_1| + c_1$ and $\tau_2 = - |\xi_2| + c_2$ and Fubini's theorem, we find that it suffices to prove
\begin{align}
& \left| \int f (\phi_{c_1}^+ (\xi_1) + \phi_{c_2}^- (\xi_2))  g_1 (\phi_{c_1}^+ (\xi_1) ) g_2 (\phi_{c_2}^-(\xi_2)) d \xi_1d\xi_2 \right| \notag \\
&\qquad \qquad \qquad \qquad \qquad \qquad  \lesssim  A^{\frac{7}{8}} \| g_1 \circ \phi_{c_1}^+\|_{L_\xi^2} \|g_2 \circ \phi_{c_2}^- \|_{L_\xi^2} \|f \|_{L_{\xi, \tau}^2}, 
\end{align}
where $f(\tau, \xi)$ is supported in $c_0 \leq \tau \pm c |\xi| \leq c_0 +1$ and 
\begin{equation*}
\phi_{c_k}^{\pm} (\xi) = (\pm |\xi|+ c_k, \xi) \quad \textnormal{for} \ k=1,2.
\end{equation*}
First we decompose $f$ by angular localization characteristic functions $\left\{ \chi_{{\mathfrak{D}}_{j}^{A_1} } \right\}_{j = -A_1}^{A_1+1}$ where 
$A_1$ is the minimal dyadic number which satisfies $A_1 \geq 2^{20} (1-c)^{-2} A$ and thickened circular 
localization characteristic functions $\left\{ \chi_{\mathbb{S}_{\delta}^{{N_0} + k \delta}} \right\}_{k= - \left[\frac{N_0}{2 \delta}\right]}^{\left[\frac{N_0}{\delta}\right]+1}$ 
where $[s]$ denotes the maximal integer which is not greater than $s \in \R$ and $\mathbb{S}_\delta^{\xi^0} = \{ (\tau, \xi ) \in \R \cross \R^2 \ | \ \xi^0 \leq |\xi|\leq \xi^0 + \delta \}$ with 
$\delta = 2^{-20} (1-c)^2 N_0 A^{-1/2}$ as follows:
\begin{equation*}
f = \sum_{k=- \left[\frac{N_0}{2 \delta}\right]}^{\left[\frac{N_0}{\delta}\right]+1} \sum_{\tiny{ j = - A_1}}^{A_1 +1 }
\chi_{\mathbb{S}_{\delta}^{N_0 + k \delta} }\chi_{{\mathfrak{D}}_{j}^{A_1}} f.
\end{equation*}
From the assumption \textnormal{(I \hspace{-0.15cm}I)}, we see that the sum of 
$(k, \, j)$ is $\sim A^{\frac{3}{4}}$. Therefore we only need to verify 
\begin{equation}
 \left| \int f  (\phi_{c_1}^+ (\xi_1) + \phi_{c_2}^- (\xi_2))  g_1 (\phi_{c_1}^+ (\xi_1) ) g_2 (\phi_{c_2}^-(\xi_2)) d \xi_1d\xi_2 \right| \lesssim  A^{\frac{1}{2}} \| g_1 \circ \phi_{c_1}^+\|_{L_\xi^2} \|g_2 \circ \phi_{c_2}^- \|_{L_\xi^2} \| f \|_{L_{\xi, \tau}^2},\label{goal-thm2.6}
\end{equation}
for $\operatorname{supp} f \subset {\mathfrak{D}}_{j}^{A_1} \cap 
\mathbb{S}_\delta^{N_0 + k \delta}$ with fixed $k \in [ - [N_0/2\delta],[N_0/\delta]+1]$, $j \in [-A_1, A_1+1]$. 
We use the scaling $(\tau, \ \xi) \to (N_0 \tau , \ N_0 \xi)$ to define
\begin{equation*}
\tilde{f} (\tau, \xi) = f (N_0 \tau, N_0 \xi), \quad \tilde{g_k} (\tau_k, \xi_k) = g_k (N_0 \tau_k, N_0 \xi_k).
\end{equation*}
If we set $\tilde{c_k} = N_0^{-1} c_k$, inequality (\ref{goal-thm2.6}) reduces to
\begin{equation}
 \left| \int \tilde{f}  (\phi_{\tilde{c_1}}^+ (\xi_1) + \phi_{\tilde{c_2}}^- (\xi_2)) 
\tilde{g_1} (\phi_{\tilde{c_1}}^+ (\xi_1) ) \tilde{g_2} (\phi_{\tilde{c_2}}^-(\xi_2)) d \xi_1d\xi_2 \right| \lesssim  A^{\frac{1}{2}} N_0^{-\frac{1}{2}} \| \tilde{g_1} 
\circ \phi_{\tilde{c_1}}^+\|_{L_\xi^2} \|\tilde{g_2}\circ \phi_{\tilde{c_2}}^- \|_{L_\xi^2} \| \tilde{f} \|_{L_{\xi, \tau}^2}.\label{goal2-thm2.6}
\end{equation}
Note that $\tilde{f}$ is supported in $S_3^\mp(N_0^{-1})$ where 
\begin{equation*}
S_3^\mp (N_0^{-1}) = \left\{ (\tau, \xi) \in {\mathfrak{D}}_{j}^{A_1} \cap \mathbb{S}_{ \tilde{\delta}}^{1 + k \tilde{\delta}} 
\ | \ \mp c |\xi| + \frac{c_0}{N_0} \leq \tau \leq \mp c |\xi| +\frac{c_0+1}{N_0} \right\}
\end{equation*}
with $\tilde{\delta} = N_0^{-1} \delta$. 
Thus from the $\ell^2$ almost orthogonality, we may assume that there exist $\xi_1^0$, $\xi_2^0$ such that
\begin{equation}
\frac{N_1}{2N_0} \leq |\xi_1^0| \leq 4 \frac{N_1}{N_0} , \quad \frac{N_2}{2 N_0} \leq  |\xi_2^0| \leq 4\frac{N_2}{N_0} \label{ele1-thm2.7}
\end{equation}
such that space variables of 
$\operatorname{supp} \tilde{g_1} \circ \phi_{\tilde{c_1}}^+$ and $\operatorname{supp} \tilde{g_2} \circ \phi_{\tilde{c_2}}^-$ are contained in the balls 
$B_{\tilde{\delta}}(\xi_1^0)$ and $B_{\tilde{\delta}}(\xi_2^0)$, respectively. By density and duality it suffices to show for continuous 
$\tilde{g_1}$ and $\tilde{g_2}$ that
\begin{equation}
\| \tilde{g_1} |_{S_1} * \tilde{g_2} |_{S_2} \|_{L^2(S_3^\pm (N_0^{-1}))} \lesssim A^{\frac{1}{2}} N_0^{-\frac{1}{2}} 
\| \tilde{g_1} \|_{L^2(S_1)} \| \tilde{g_2} \|_{L^2(S_2)}\label{goal3-thm2.7}
\end{equation}
where $S_1$, $S_2$ denote the following surfaces 
\begin{align*}
S_1 = \{ \phi_{\tilde{c_1}}^+ (\xi_1)\in \R^3 \ | \ \xi_1 \in B_{\tilde{\delta}}(\xi_1^0) \}, \\
S_2 = \{ \phi_{\tilde{c_2}}^- (\xi_2)\in \R^3 \ | \ \xi_2 \in B_{\tilde{\delta}}(\xi_2^0) \}.
\end{align*}
(\ref{goal3-thm2.7}) is immediately established from
\begin{equation}
\| \tilde{g_1} |_{S_1} * \tilde{g_2} |_{S_2} \|_{L^2(S_3^\pm)} \lesssim A^{\frac{1}{2}}  
\| \tilde{g_1} \|_{L^2(S_1)} \| \tilde{g_2} \|_{L^2(S_2)}\label{goal4-thm2.7}
\end{equation}
where 
\begin{equation*}
S_3^\mp = \{ (\psi^\mp (\xi), \xi) \in {\mathfrak{D}}_{j}^{A_1} \cap \mathbb{S}_{ \tilde{\delta}}^{1 + k \tilde{\delta}} \ | \ \psi^\mp(\xi) =  \mp c |\xi|+ \frac{c_0}{N_0} \}.
\end{equation*}
For any $\sigma_i \in S_i$, $i=1,2,3$, there exist $\xi_1$, $\xi_2$, $\xi$ such that
\begin{equation*}
\sigma_1=\phi_{\tilde{c_1}}^+ (\xi_1), \quad \sigma_2 =  \phi_{\tilde{c_2}}^+ (\xi_2), \quad \sigma_3 = (\psi (\xi), \xi),
\end{equation*}
and the unit normals $\mathfrak{n}_i$ on $\sigma_i$ are written as
\begin{align*}
& \mathfrak{n}_1(\sigma_1) = \frac{1}{\sqrt{2}} \left(-1, \ \frac{(\xi_1)_1}{|\xi_1|}, \ \frac{(\xi_1)_2}{|\xi_1|} \right), \\
& \mathfrak{n}_2 (\sigma_2) = \frac{1}{\sqrt{2}} \left(1, \ \frac{(\xi_2)_1}{|\xi_2|}, \ \frac{(\xi_2)_2}{|\xi_2|} \right), \\
& \mathfrak{n}_3 (\sigma_3) = \frac{1}{\sqrt{c^2 +1}} \left( \pm 1, \ c\frac{(\xi)_1}{|\xi|}, \  c\frac{(\xi)_2}{|\xi|} \right).
\end{align*}
We deduce from $1 \lesssim |\xi|$ and \eqref{ele1-thm2.7} that the surfaces $S_1$, $S_2$, $S_3^\mp$ satisfy the following 
H\"{o}lder condition.
\begin{align}
\sup_{\sigma_i, \sigma_i' \in S_i} \frac{|\mathfrak{n}_i(\sigma_i) - 
\mathfrak{n}_i(\sigma_i')|}{|\sigma_i - \sigma_i'|}
+ \frac{|\mathfrak{n}_i(\sigma_i) (\sigma_i - \sigma_i')|}{|\sigma_i - \sigma_i'|^2} \lesssim 1,\label{aiueo1}\\
\sup_{\sigma_3, \sigma_3' \in S_3^\pm} \frac{|\mathfrak{n}_3(\sigma_3) - 
\mathfrak{n}_3(\sigma_3')|}{|\sigma_3 - \sigma_3'|}
+ \frac{|\mathfrak{n}_3 (\sigma_3) (\sigma_3 - \sigma_3')|}{|\sigma_3 - \sigma_3'|^2} \lesssim 1\label{aiueo2}.
\end{align}
We may assume that there exist $\xi_1', \xi_2', \xi' \in \R^2$ such that
\begin{equation*}
\xi_1' + \xi_2' = \xi', \quad \phi_{\tilde{c_1}}^+ (\xi_1') \in S_1, \ \phi_{\tilde{c_2}}^- (\xi_2') \in S_2, \ (\psi^\mp (\xi'), \xi') \in S_3^\mp,
\end{equation*}
otherwise the left-hand side of (\ref{goal3-thm2.7}) vanishes. 
Let $\sigma_1' = \phi_{\tilde{c_1}}^+(\xi_1')$, $\sigma_2' = \phi_{\tilde{c_2}}^- (\xi_2')$, $\sigma_3' =  (\psi^\mp(\xi'), \xi') $. 
For any $\sigma_1 = \phi_{\tilde{c_1}}^+(\xi_1) \in S_1$, we deduce from $\xi_1$, $\xi_1' \in B_{\tilde{\delta}}(\xi_1^0)$ 
and $A \leq M_1 \leq 2^{10} (1-c)^{-1} N_1/N_0$  that
\begin{equation}
|\mathfrak{n}_1(\sigma_1) - \mathfrak{n}_1(\sigma_1')| \leq 2^{-18} \frac{N_0}{N_1} (1-c)^2 A^{-\frac{1}{2}} \leq 2^{-8} (1-c) A^{-\frac{3}{2}}.\label{elea-thm2.7}
\end{equation}
Similarly, for any $\sigma_2 = \phi_{\tilde{c_2}}^- (\xi_2) \in S_2$ we have
\begin{equation}
|\mathfrak{n}_2(\sigma_2) - \mathfrak{n}_2(\sigma_2')| \leq 2^{-18} \frac{N_0}{N_2} (1-c)^2 A^{-\frac{1}{2}} \leq 2^{-8} (1-c) A^{-\frac{3}{2}}.\label{eleb-thm2.7}
\end{equation}
For any $\sigma_3 \in S_3^\mp$, it follows from $S_3^\mp \subset {\mathfrak{D}}_{j}^{A_1} $ that 
\begin{equation}
|\mathfrak{n}_3(\sigma_3) - \mathfrak{n}_3(\sigma_3')| \leq 2^{-10} (1-c) A^{-1}.\label{elec-thm2.7}
\end{equation}
It is obvious that $|\sigma_1- \sigma_1'|$, $|\sigma_2- \sigma_2'| \leq 2 \tilde{\delta} \leq 2^{-10} (1-c)^{2}A^{-1/2}$ holds, then we get from 
(\ref{elea-thm2.7}) and (\ref{eleb-thm2.7}) that
\begin{align}
|(\sigma_1 - \sigma_1') \cdot \mathfrak{n}_1(\sigma_1')| \leq 2^{-15} (1-c)^2 A^{-2},\label{eleaa-thm2.7}\\
|(\sigma_2 - \sigma_2') \cdot \mathfrak{n}_2(\sigma_2')| \leq 2^{-15} (1-c)^2 A^{-2}.\label{eleab-thm2.7}
\end{align}
Similarly, we deduce from $\left|\sigma_3 - \frac{|\sigma_3|}{|\sigma_3'|} \sigma_3'\right| \leq 2^{-10}(1-c)^2 A^{-1}$ and 
(\ref{elec-thm2.7}) that
\begin{equation}
|(\sigma_3 - \sigma_3') \cdot \mathfrak{n}_3 (\sigma_3')| = 
\left| \left( \sigma_3 - \frac{|\sigma_3|}{|\sigma_3'|} \sigma_3' \right) \cdot \mathfrak{n}_3(\sigma_3') \right| \leq 
2^{-15} (1-c)^2 A^{-2}.\label{eleac-thm2.7}
\end{equation}
(\ref{eleaa-thm2.7}) means that $S_1$ is contained in an slab of thickness $2^{-15} (1-c)^2 A^{-2}$ with respect to the 
$\mathfrak{n}_1(\sigma_1')$ direction. 
From $\ell^2$ orthogonality, we may assume that $S_2$ and $S_3$ are also contained in similar $2^{-15} (1-c)^2 A^{-2}$ thick slabs;
\begin{align*}
|(\sigma_2 - \sigma_2') \cdot \mathfrak{n}_1 (\sigma_1')| \leq 2^{-15} (1-c)^2 A^{-2},\\
|(\sigma_3 - \sigma_3') \cdot \mathfrak{n}_1 (\sigma_1')| \leq 2^{-15} (1-c)^2 A^{-2}.
\end{align*}
Similarly, we may assume that surfaces $S_1$, $S_2$ are contained in slabs of thickness $2^{-15} (1-c)^2 A^{-2}$ with respect to the 
$\mathfrak{n}_2(\sigma_2')$ direction and the surfaces $S_1$, $S_2$ are contained in slabs of thickness $2^{-15} (1-c)^2 A^{-2}$ 
with respect to the $\mathfrak{n}_3(\sigma_3')$ direction. Collection the above assumptions, for $i,j=1,2,3$, 
\begin{equation}
|(\sigma_i - \sigma_i') \cdot \mathfrak{n}_j (\sigma_j')| \leq 2^{-15} (1-c)^2 A^{-2}.\label{ele01-thm2.7}
\end{equation}
We define $T \ : \ \R^3 \to \R^3$ as
\begin{equation*}
T= 2^{-10} (1-c)^{2} A^{-2}(N^\top)^{-1}, \qquad N = N(\sigma_1', \sigma_2', \sigma_3').
\end{equation*}
If the following conditions are established, we immediately obtain the desired estimate (\ref{goal4-thm2.7}) by applying Proposition \ref{prop2.7} with $T$ and $\tilde{S}_i := T^{-1} S_i \ (i = 1,2,3)$.
\begin{align*}
\textnormal{(I)} & \quad \frac{1-c}{2} A^{-1} \leq |\textnormal{det} N(\sigma_1, \sigma_2, \sigma_3)| \quad 
\textnormal{for any} \ \sigma_i \in S_i.\\
\textnormal{(I \hspace{-0.15cm}I)} & \quad \textnormal{diam}(\tilde{S_i}) <1.\\
\textnormal{(I \hspace{-0.16cm}I \hspace{-0.16cm}I)} &  \quad \frac{1}{2} \leq \textnormal{det} (\tilde{\mathfrak{n}_1}(\tilde{\sigma_1}), \tilde{\mathfrak{n}_2}(\tilde{\sigma_2}), 
\tilde{\mathfrak{n}_3}(\tilde{\sigma_3})) \leq 1 \quad \textnormal{for any} \ \tilde{\sigma_i} \in \tilde{S_i}.\\
\textnormal{(I \hspace{-0.18cm}V)} & \sup_{\tilde{\sigma}_i, \tilde{\sigma}_i^0 \in \tilde{S_i}} 
\frac{|\tilde{\mathfrak{n}_i}(\tilde{\sigma_i}) - \tilde{\mathfrak{n}_i}(\tilde{\sigma}_i^0)|}{|\tilde{\sigma_i}-\tilde{\sigma}_i^0|} 
+ \frac{|\tilde{\mathfrak{n}_i}(\tilde{\sigma}_i^0) \cdot (\tilde{\sigma_i}-\tilde{\sigma}_i^0) |}{|\tilde{\sigma_i}-\tilde{\sigma}_i^0|^2} \leq 1 \ 
\textnormal{for the unit normals} \ \tilde{\mathfrak{n}}_i \ \textnormal{on} \ \tilde{S}_i.
\end{align*}
We first show (I). From (\ref{elea-thm2.7})-(\ref{elec-thm2.7}) it suffices to show
\begin{equation}
(1-c)A^{-1} \leq |\textnormal{det} N(\sigma_1', \sigma_2', \sigma_3')| .
\end{equation}
Seeing that $\sigma_1' = \phi_{\tilde{c_1}}^+(\xi_1')$, $\sigma_2' = \phi_{\tilde{c_2}}^- (\xi_2')$, $\sigma_3' =  (\psi^\mp(\xi'), \xi') $ and 
$\xi_1'+\xi_2' = \xi'$, we get
\begin{align}
|\textnormal{det} N(\sigma_1', \sigma_2', \sigma_3')| \geq & 
\frac{1}{4} \left|\textnormal{det}
\begin{pmatrix}
-1 & 1 & \pm 1 \\
\frac{(\xi_1')_1}{|\xi_1'|}  & \frac{(\xi_2')_1}{|\xi_2'|} & c\frac{(\xi')_1}{|\xi'|} \\
\frac{(\xi_1')_2}{|\xi_1'|}  & \frac{(\xi_2')_2}{|\xi_2'|} & c\frac{(\xi')_2}{|\xi'|}
\end{pmatrix} \right| \notag \\
\geq & \frac{1}{4} \left| \frac{(\xi_1')_1 (\xi_2')_2 - (\xi_1')_2 (\xi_2')_2}{|\xi_1'||\xi_2'|} \right| 
\left(1 - c \left| \frac{|\xi_2'|}{|\xi'|} - \frac{|\xi_1'|}{|\xi'|} \right| \right) \notag\\
\geq & (1-c) A^{-1}.\label{ele002-thm2.7}
\end{align}
\textnormal{(I \hspace{-0.15cm}I)} is established from (\ref{ele01-thm2.7}).
\begin{align*}
|T^{-1} (\sigma_i - \sigma_i)| & =  2^{10} (1-c)^{-2} A^{2} \left|
\begin{pmatrix}
\mathfrak{n}_1  (\sigma_1') \cdot (\sigma_i - \sigma_i') \\
\mathfrak{n}_2  (\sigma_2') \cdot (\sigma_i - \sigma_i') \\
\mathfrak{n}_3  (\sigma_3') \cdot (\sigma_i - \sigma_i') 
\end{pmatrix} \right|\\
& \leq 2^{-3}< \frac{1}{2}.
\end{align*}
Next we show (I \hspace{-0.16cm}I \hspace{-0.16cm}I). Note that the unit normals $\tilde{\mathfrak{n}}_i$ on $\tilde{S}_i$ are written as follows.
\begin{equation*}
\tilde{\mathfrak{n}}_i (\tilde{\sigma_i}) =  \frac{(T^{-1})^\top \mathfrak{n}_i(T \tilde{\sigma_i}) }{|(T^{-1})^\top \mathfrak{n}_i(T \tilde{\sigma_i})|}
=  \frac{N^{-1} \mathfrak{n}_i(T \tilde{\sigma_i}) }{|N^{-1} \mathfrak{n}_i(T \tilde{\sigma_i})|}.
\end{equation*}
In particular, the unit normals on $T^{-1} \sigma_i'$  are the unit vectors $e_i$;
\begin{equation}
 \tilde{\mathfrak{n}}_i (T^{-1} \sigma_i') = N^{-1} \mathfrak{n}_i (\sigma_i') = e_i.\label{ele00a-thm2.7}
\end{equation}
From (\ref{ele002-thm2.7}), we get
\begin{equation}
\| N^{-1} \| = \| (N^\top)^{-1} \| \leq 2 |\textnormal{det} N^\top|^{-1} \|N^\top \|^2 \leq 12 (1-c)^{-1}A.\label{ele003-thm2.7}
\end{equation}
Thus we obtain
\begin{equation}
\|T \| \leq 2^{-6} (1-c) A^{-1}.\label{ele004-thm2.7}
\end{equation}
We deduce from (\ref{elea-thm2.7})-(\ref{elec-thm2.7}), (\ref{ele00a-thm2.7}), (\ref{ele003-thm2.7}) that
\begin{equation}
|N^{-1} \mathfrak{n}_i(T \tilde{\sigma_i}) - e_i| = |N^{-1} (\mathfrak{n}_i(T \tilde{\sigma_i}) - \mathfrak{n}_i (\sigma_i') )| \leq 2^{-7}.\label{ele005-thm2.7}
\end{equation}
This gives $|\tilde{\mathfrak{n}}_i(\tilde{\sigma_i}) - e_i|  \leq 2^{-5}$ and (I \hspace{-0.16cm}I \hspace{-0.16cm}I) is now obtained. 
Finally we show (I \hspace{-0.18cm}V). It follows from (\ref{ele003-thm2.7})-(\ref{ele005-thm2.7}) that
\begin{align*}
\frac{|\tilde{\mathfrak{n}_i}(\tilde{\sigma_i}) - \tilde{\mathfrak{n}_i}(\tilde{\sigma}_i^0)|}{|\tilde{\sigma_i}-\tilde{\sigma}_i^0|} 
& \leq 3 \frac{|N^{-1} ( \mathfrak{n}_i(T \tilde{\sigma_i}) - {\mathfrak{n}_i}(T \tilde{\sigma}_i^0))|}{|\tilde{\sigma_i}-\tilde{\sigma}_i^0|} \\
& \leq 3 \|N^{-1}\| \|T \|  \frac{|\mathfrak{n}_i(T \tilde{\sigma_i}) - {\mathfrak{n}_i}(T \tilde{\sigma}_i^0)|}{|T\tilde{\sigma_i}-T\tilde{\sigma}_i^0|} 
\lesssim 1.
\end{align*}
The last inequality is verified from \eqref{aiueo1} and \eqref{aiueo2}. 
Similarly, from (\ref{ele004-thm2.7}) and $ (T^{-1})^\top N^{-1} = 2^{10} (1-c)^{-2} A^2 E$ we have 
\begin{align*}
\frac{|\tilde{\mathfrak{n}_i}(\tilde{\sigma}_i^0) \cdot (\tilde{\sigma_i}-\tilde{\sigma}_i^0) |}{|\tilde{\sigma_i}-\tilde{\sigma}_i^0|^2}
& \leq 2 \|T \|^2 \frac{|N^{-1}{\mathfrak{n}_i}(T \tilde{\sigma}_i^0) \cdot (T^{-1} T\tilde{\sigma_i}- 
T^{-1} T \tilde{\sigma}_i^0) |}{|T \tilde{\sigma_i}- T \tilde{\sigma}_i^0|^2}\\
& \leq 2 \|T \|^2 \frac{| (T^{-1})^\top N^{-1}{\mathfrak{n}_i}(T \tilde{\sigma}_i^0) \cdot (T\tilde{\sigma_i}- 
 T \tilde{\sigma}_i^0) |}{|T \tilde{\sigma_i}- T \tilde{\sigma}_i^0|^2}\\
& \leq \frac{1}{2}  \frac{|{\mathfrak{n}_i}(T \tilde{\sigma}_i^0) \cdot (T\tilde{\sigma_i}- 
 T \tilde{\sigma}_i^0) |}{|T \tilde{\sigma_i}- T \tilde{\sigma}_i^0|^2} \lesssim 1.
\end{align*}
This completes  (I \hspace{-0.18cm}V).
\end{proof}
We now consider $0 \leq \angle (\xi_1, \ \xi_2) \leq \pi/2$. First we show the estimate 
which is similar to Proposition \ref{thm2.6} for 
$64 \leq A \leq N_0^{\frac{1}{2}}$ and $16 \leq |j_1 - j_2|\leq 32$. In this case, thanks to $0 \leq \angle (\xi_1, \ \xi_2) \leq \pi/2$, 
$N_0 \sim N_1 \sim N_2$ always holds true and we can obtain the better estimates compared to 
Proposition \ref{thm2.6}.
\begin{prop}\label{thm2.8}
Let $\tau = \tau_1 + \tau_2$, $\xi = \xi_1 +\xi_2$, $0<c<1$ and $f$, $g_1$, $g_2 \in L^2$ be satisfy
\begin{equation*}
\operatorname{supp} f \subset K^{\pm,c}_{N_0,L_0}, \quad \operatorname{supp} g_1 \subset {\mathfrak{D}}_{j_1}^A \cap K^{-}_{N_1,L_1}, \quad 
 \operatorname{supp} g_2 \subset {\mathfrak{D}}_{j_2}^A \cap K^{+}_{N_2,L_2},
\end{equation*}
and $N_0 \sim N_1 \sim N_2$, \ $64 \leq A \leq N_0^{\frac{1}{2}} $, \ $16 \leq |j_1 - j_2 |\leq 32$. 
Then the following estimate holds:
\begin{equation}
\left| \int f(\tau, \xi) g_1(\tau_1, \xi_1) g_2 (\tau_2, \xi_2) d\tau_1 d\tau_2 d\xi_1 d \xi_2 \right|
\lesssim A^{\frac{1}{2}} (L_0 L_1 L_2)^\frac{1}{2} 
\|f\|_{{L_{\xi,\tau}^2}} \|g_1 \|_{{L_{\xi,\tau}^2}} \|g_2 \|_{{L_{\xi,\tau}^2}}.\label{goal000-thm2.8}
\end{equation}
\end{prop}
\begin{proof}
The proof is almost analogous to that of Proposition \ref{thm2.6}. 
Difference between them is a step of decomposition. Precisely, in the proof of 
Proposition \ref{thm2.6}, we decomposed $f$ into $\sim A^\frac{3}{4}$ pieces. We here decompose functions into finitely many pieces. 
From $\operatorname{supp} g_1 \subset {\mathfrak{D}}_{j_1}^A$, $\operatorname{supp} g_2 \subset {\mathfrak{D}}_{j_2}^A$ and $16 \leq |j_1 - j_2 |\leq 32$, after 
suitable and harmless decomposition, we can assume that 
there exists $j$ such that $16 \leq |j_1 - j | \leq 32$ and $\operatorname{supp} f \in {\mathfrak{D}}_{j}^A$. Furthermore we decompose $f$, $g_1$, $g_2$ into 
finitely many pieces as follows:
\begin{equation*}
f=\sum_{j'=j^0}^{j^0+k} \chi_{{\mathfrak{D}}_{j'}^{A_1}}f, \qquad 
g_1=\sum_{j_1'=j_1^0}^{j_1^0+k} \chi_{{\mathfrak{D}}_{j_1'}^{A_1}}g_1, \qquad
g_2=\sum_{j_2' =j_2^0}^{j_2^0+k} \chi_{{\mathfrak{D}}_{j_2'}^{A_1}}g_2
\end{equation*}
where $k$ is the minimal dyadic number which satisfies $k \geq 2^{20} (1-c)^{-2}$, $A_1 := k A$ and $j^0, 
j_1^0, j_2^0$ satisfy
\begin{equation*}
\bigcup_{j^0 \leq j' \leq j^0+k} {\mathfrak{D}}_{j'}^{A_1} = {\mathfrak{D}}_{j}^A, \qquad
\bigcup_{j_1^0 \leq j_1' \leq j_1^0 +k} {\mathfrak{D}}_{j_1'}^{A_1} = {\mathfrak{D}}_{j_1}^A, \qquad
\bigcup_{j_2^0 \leq j_2' \leq j_2^0 +k} {\mathfrak{D}}_{j_2'}^{A_1} = {\mathfrak{D}}_{j_2}^A.
\end{equation*}
Thanks to the finiteness of $k$, it suffices to prove the desired estimate (\ref{goal000-thm2.8}) for
\begin{equation*}
\operatorname{supp} f \subset {\mathfrak{D}}_{j'}^{A_1}, \quad \operatorname{supp} g_1 \subset {\mathfrak{D}}_{j_1'}^{A_1}, \quad \operatorname{supp} g_2 \subset {\mathfrak{D}}_{j_2'}^{A_1}
\end{equation*}
with fixed $j' \in [j^0, j^0+k]$, $j_1' \in [j_1^0, j_1^0 +k]$, $j_2' \in [j_2^0, j_2^0+k]$. 

We utilize the same notations as in the proof of Proposition \ref{thm2.6}. By the same argument as of 
the proof of Proposition \ref{thm2.6}, we only need to verify the following estimate:
\begin{equation}
\| \tilde{g_1} |_{S_1} * \tilde{g_2} |_{S_2} \|_{L^2(S_3)} \lesssim A^{\frac{1}{2}}  
\| \tilde{g_1} \|_{L^2(S_1)} \| \tilde{g_2} \|_{L^2(S_2)}\label{goal00a-thm2.8}
\end{equation}
where 
\begin{align*}
S_1 &  = \left\{ \phi_{\tilde{c_1}}^+ (\xi_1)\in {\mathfrak{D}}_{j_1'}^{A_1} \ | \ \frac{1-c}{4} \leq |\xi_1| \leq 2  \right\}, \quad 
S_2  = \left\{ \phi_{\tilde{c_2}}^- (\xi_2)\in {\mathfrak{D}}_{j_2'}^{A_1} \ | \ \frac{1-c}{4} \leq |\xi_2| \leq 2 \right\},\\
S_3 & = \left\{ (\psi^\mp (\xi), \xi) \in {\mathfrak{D}}_{j'}^{A_1} \ | \ \frac{1}{2} \leq |\xi| \leq 4 , \  \psi^\mp(\xi) =  \mp c |\xi|+ \frac{c_0}{N_0} \right\}.
\end{align*}
We recall that the unit normals on $\sigma_i \in S_i$ ($i=1,2,3$) are written as;
\[
 \mathfrak{n}_1(\sigma_1) = \frac{1}{\sqrt{2}} \Bigl(-1, \ \frac{(\xi_1)_1}{|\xi_1|}, \ \frac{(\xi_1)_2}{|\xi_1|} \Bigr), \ \ 
 \mathfrak{n}_2 (\sigma_2) = \frac{1}{\sqrt{2}} \Bigl(1, \ \frac{(\xi_2)_1}{|\xi_2|}, \ \frac{(\xi_2)_2}{|\xi_2|} \Bigr), \ \ 
 \mathfrak{n}_3 (\sigma_3) = \frac{1}{\sqrt{c^2 +1}} \Bigl( \pm 1, \ c\frac{(\xi)_1}{|\xi|}, \  c\frac{(\xi)_2}{|\xi|} \Bigr).
\]
where 
\begin{equation*}
\sigma_1=\phi_{\tilde{c_1}}^+ (\xi_1), \quad \sigma_2 =  \phi_{\tilde{c_2}}^+ (\xi_2), \quad \sigma_3 = (\psi (\xi), \xi).
\end{equation*}
We may assume that there exist $\xi_1', \xi_2', \xi' \in \R^2$ such that
\begin{equation*}
\xi_1' + \xi_2' = \xi', \quad (\sigma_1' :=) \phi_{\tilde{c_1}}^+ (\xi_1') \in S_1, \ 
(\sigma_2' := ) \phi_{\tilde{c_2}}^- (\xi_2') \in S_2, \ (\sigma_3' :=)  (\psi^\mp (\xi'), \xi') \in S_3.
\end{equation*}
From $S_1\subset {\mathfrak{D}}_{j_1'}^{A_1}$, $S_2 \subset {\mathfrak{D}}_{j_2'}^{A_1}$ and $S_3 \subset {\mathfrak{D}}_{j'}^{A_1}$, we easily observe
\begin{align}
|\mathfrak{n}_1(\sigma_1) - \mathfrak{n}_1(\sigma_1')| \leq 2^{-10} (1-c) A^{-1},\label{elea-thm2.8}\\ 
|\mathfrak{n}_2(\sigma_2) - \mathfrak{n}_2(\sigma_2')| \leq 2^{-10} (1-c) A^{-1},\label{eleb-thm2.8} \\
|\mathfrak{n}_3(\sigma_3) - \mathfrak{n}_3(\sigma_3')| \leq 2^{-10} (1-c) A^{-1}.\label{elec-thm2.8}
\end{align}
The above estimates (\ref{elea-thm2.8})-(\ref{elec-thm2.8}) give 
\begin{align*}
|(\sigma_1 - \sigma_1') \cdot \mathfrak{n}_1 (\sigma_1')| = 
\left| \left( \sigma_1 - \frac{|\sigma_1|}{|\sigma_1'|} \sigma_1' \right) \cdot \mathfrak{n}_1(\sigma_1') \right| \leq 
2^{-20} (1-c)^2 A^{-2},\\
|(\sigma_2 - \sigma_2') \cdot \mathfrak{n}_2 (\sigma_2')| = 
\left| \left( \sigma_2 - \frac{|\sigma_2|}{|\sigma_2'|} \sigma_2' \right) \cdot \mathfrak{n}_2(\sigma_2') \right| \leq 
2^{-20} (1-c)^2 A^{-2},\\
|(\sigma_3 - \sigma_3') \cdot \mathfrak{n}_3 (\sigma_3')| = 
\left| \left( \sigma_3 - \frac{|\sigma_3|}{|\sigma_3'|} \sigma_3' \right) \cdot \mathfrak{n}_3 (\sigma_3') \right| \leq 
2^{-20} (1-c)^2 A^{-2}.
\end{align*}
By the same argument as in the proof of Proposition \ref{thm2.6}, we can assume
\begin{equation}
|(\sigma_i - \sigma_i') \cdot \mathfrak{n}_j (\sigma_j')| \leq 2^{-20} (1-c)^2 A^{-2} \quad \textnormal{for any} \ i,  j = 1,2,3.\label{ele01-thm2.8}
\end{equation}
The remaining part is only to prove (I)-(I \hspace{-0.18cm}V) in Proposition \ref{thm2.6} with 
\begin{equation*}
T= 2^{-10} (1-c)^{2} A^{-2}(N^\top)^{-1}, \qquad N = N(\sigma_1', \sigma_2', \sigma_3')
\end{equation*}
and $\tilde{S}_i := T^{-1} S_i \ (i = 1,2,3)$. (I)-(I \hspace{-0.18cm}V) are verified from  (\ref{elea-thm2.8})-(\ref{ele01-thm2.8}) as we 
proved in the proof of Proposition \ref{thm2.6}. To avoid redundancy, we omit the proof of them.
\end{proof}
Lastly, we consider the case of sufficiently small $\angle (\xi_1, \xi_2)$. 
\begin{prop}\label{prop2.10}
Let $\tau = \tau_1 + \tau_2$, $\xi = \xi_1 +\xi_2$, $0<c<1$ and $M_0$ is the minimal dyadic number which satisfies 
$N_0^{\frac{1}{2}} \leq M_0$. 
We assume that $f$, $g_1$, $g_2 \in L^2$ satisfy
\begin{equation*}
\operatorname{supp} f \subset K^{\pm,c}_{N_0,L_0}, \quad \operatorname{supp} g_1 \subset {\mathfrak{D}}_{j_1}^{M_0} \cap K^{-}_{N_1,L_1}, \quad 
 \operatorname{supp} g_2 \subset {\mathfrak{D}}_{j_2}^{M_0}\cap K^{+}_{N_2,L_2},
\end{equation*}
with $N_0 \sim N_1 \sim N_2$, \ $ |j_1 - j_2 |\leq 16$. 
Then the following estimate holds;
\begin{equation}
\left| \int f(\tau, \xi) g_1(\tau_1, \xi_1) g_2 (\tau_2, \xi_2) d\tau_1 d\tau_2 d\xi_1 d \xi_2 \right| 
\lesssim N_0^{\frac{1}{4}}  (L_0 L_{\textnormal{min}}^{12})^\frac{1}{2} 
\|f\|_{{L_{\xi,\tau}^2}} \|g_1 \|_{{L_{\xi,\tau}^2}} \|g_2 \|_{{L_{\xi,\tau}^2}}.\label{goal001-prop2.9}
\end{equation}
\end{prop}
\begin{proof}
We can assume $L_1 \leq L_2$ by symmetry. 
By H\"{o}lder inequality, (\ref{goal001-prop2.9}) is established if we show
\begin{equation}
 \| \chi_{K^{+}_{N_2,L_2}} \left( \left( 
\chi_{K^{\pm_0, c}_{N_0,L_0}}f \right) * \left( \chi_{K^{\pm_1}_{N_1,L_1}}g
\right) \right) \|_{L_{\xi, \tau}^2} \lesssim (N_0^{\frac{1}{4}} L_0 L_1 )^{1/2} \|f\|_{ L_{\xi,\tau}^2} 
\|g \|_{ L_{\xi,\tau}^2}\label{goal002-prop2.9}
\end{equation}
regardless of the choice of the signs $\pm_0$, $\pm_1$. 
It is easily confirmed that (\ref{goal002-prop2.9}) can be verified by the proof of Proposition \ref{thm2.2} with minor modification. 
Indeed, same as in the proof of Proposition \ref{thm2.2}, we find that the desired estimate (\ref{goal002-prop2.9}) is shown by
\begin{equation}
\sup_{\tau,\xi} |E(\tau, \xi) | \lesssim N_0^{\frac{1}{2}} L_0 L_1\label{goal003-prop2.9}
\end{equation}
where $E(\tau, \xi) 
 := \{(\tau_1, \xi_1) \in  {\mathfrak{D}}_{0}^{M_0} \ | \ \LR{\tau - \tau_1 \pm_0 c |\xi - \xi_1|} \sim L_0, \LR{\tau_1 \pm_1 |\xi_1|} \sim L_1 \}.$ 
Applying the same proof as in Proposition \ref{thm2.2}, we immediately obtain (\ref{goal003-prop2.9}) thanks to 
$N_1 M_0^{-1} \sim N_0^{\frac{1}{2}}.$
\end{proof}
We now prove the crucial estimates (\ref{goal-1-2}) and (\ref{goal-2-2}) with $\pm_2 = +$ and $N_0 \lesssim N_1 \sim N_2$. 
\begin{thm}
Let $0<c<1$. 
For any $s \in (-3/4 , \, 0)$, there exists $b \in (1/2, \, 1)$ such that for $f,g_1, g_2 \in \mathcal{S}(\R \cross \R^2)$, the following estimates hold;
\begin{align}
& \sum_{N_1} \sum_{1\leq N_0 \lesssim N_1 \sim N_2}\sum_{L_j} \ I_1^+ \lesssim \|f \|_{\ha{X}^{s,\,b}_{\pm,\,c}} \|g_1 \|_{\ha{X}^{s,\,b}_{-}}\|g_2 \|_{\ha{X}^{-s,\,1-b-\e}_{+}},\label{goal001-thm2.10}\\
& \sum_{N_1} \sum_{1\leq N_0 \lesssim N_1 \sim N_2} \sum_{L_j} \ I_2^+ 
\lesssim \|f \|_{\ha{X}^{-s,\,1-b-\e}_{\pm,\,c}} \|g_1 \|_{\ha{X}^{s,\,b}_{-}}\|g_2 \|_{\ha{X}^{s,\,b}_{+}},\label{goal002-thm2.10}
\end{align}
where
\begin{align*}
I_1^+ := & \left|N_1^{-1} \int{ \left( \chi_{K^{\pm,c}_{N_0,L_0}} f \right) 
\left(
(\chi_{K^{-}_{N_1,L_1}} g_1) * (\chi_{K^{+}_{N_2,L_2}} g_2)
\right)} 
d\tau d\xi \right|,\\
I_2^+ := & \left|N_0 N_1^{-1}N_2^{-1}  \int{ \left( \chi_{K^{\pm,c}_{N_0,L_0}} f \right) 
\left(
(\chi_{K^{-}_{N_1,L_1}} g_1) * (\chi_{K^{+}_{N_2,L_2}} g_2)
\right)} 
d\tau d\xi \right|.
\end{align*}
\end{thm}
\begin{proof}
We first note that if $N_1 \lesssim L_{\textnormal{max}}^{012}$ then (\ref{goal001-thm2.10}) and (\ref{goal002-thm2.10}) are obtained 
by the same proof as that of Theorem \ref{thm2.3}. Therefore we can assume $ L_{\textnormal{max}}^{012} \lesssim N_1$. 
We can also assume that $1 \ll N_0$. Indeed, if $N_0 \sim 1$ \eqref{goal001-thm2.10} and \eqref{goal002-thm2.10} are immediately 
obtained by using Proposition \ref{thm2.1} as $N_0 \sim1$. 

If $s \in (-3/4 , \, - 1/2)$, considering $N_0 \lesssim N_1 \sim N_2$, we observe that the latter estimate (\ref{goal002-thm2.10}) is difficult to show compared with the former one. Clearly, the proof of (\ref{goal001-thm2.10}) and (\ref{goal002-thm2.10}) 
become easier as $s$ gets greater. Therefore, we here focus our attention on proving  (\ref{goal002-thm2.10}) for $s \in (-3/4 , \, - 1/2)$. 

(\ref{goal002-thm2.10}) is equivalent to
\begin{equation}\label{goal003-thm2.10}
\begin{split}
\sum_{N_1} \sum_{N_0 \lesssim N_1 \sim N_2}  \sum_{L_j \lesssim N_1} 
N_0 N_1^{-2} 
 & \left| \int{ {f^{\pm, c}} (\tau_1 + \tau_2, \xi_1 + \xi_2) {g_1^-}(\tau_1, \xi_1) {g_2^+} (\tau_2, \xi_2) }d\tau_1 d\tau_2 d\xi_1 d \xi_2 \right|\\
&  
\lesssim \|f \|_{\ha{X}^{-s,\,1-b-\e}_{\pm,\,c}} \|g_1 \|_{\ha{X}^{s,\,b}_{-}}\|g_2 \|_{\ha{X}^{s,\,b}_{+}}. 
\end{split}
\end{equation}
Here we utilized the denotations ${f^{\pm, c}} := \chi_{K^{\pm, c}_{N_0,L_0}} f$, ${g_1^-} := \chi_{{K^{-}_{N_1,L_1}}}g_1$, ${g_2^+} := \chi_{{K^{+}_{N_2,L_2}}}g_2$.
For simplicity, we use 
\begin{equation*}
I(f, g, h) := N_0 N_1^{-2} \left| \int{ f (\tau, \xi) g(\tau_1, \xi_1) h (\tau_2, \xi_2) }d\tau_1 d\tau_2 d\xi_1 d \xi_2 \right|
\end{equation*}
where $\tau = \tau_1+\tau_2$ and $\xi =\xi_1 +\xi_2$.
By the decomposition of $\R^3 \cross \R^3$
\begin{align*}
\R^3 \cross \R^3 = & \bigcup_{\tiny{\substack{-M_0 \leq j_1,j_2 \leq M_0 -1\\|j_1 - j_2|\leq 16}}} 
{\mathfrak{D}}_{j_1}^{M_0} \cross {\mathfrak{D}}_{j_2}^{M_0} \cup \bigcup_{64 \leq A \leq M_0} \ \bigcup_{\tiny{\substack{-A \leq j_1,j_2 \leq A -1\\ 16 \leq |j_1 - j_2|\leq 32}}} 
{\mathfrak{D}}_{j_1}^A \cross {\mathfrak{D}}_{j_2}^A \\
\cup & \bigcup_{\tiny{\substack{-M_1 \leq j_1,j_2 \leq M_1 -1\\|j_1 - j_2 \pm M_1|\leq 16}}} 
{\mathfrak{D}}_{j_1}^{M_1} \cross {\mathfrak{D}}_{j_2}^{M_1} \cup \bigcup_{64 \leq A \leq M_1} \ 
\bigcup_{\tiny{\substack{-A \leq j_1,j_2 \leq A +1\\ 16 \leq |j_1 - j_2 \pm A |\leq 32}}} 
{\mathfrak{D}}_{j_1}^A \cross {\mathfrak{D}}_{j_2}^A.
\end{align*}
where $M_0$ and $M_1$ are the minimal dyadic number which satisfies respectively
\begin{equation*}
N_0^\frac{1}{2} \leq M_0 , \qquad 2^{7} (1-c)^{-\frac{1}{2}} \frac{(N_1 N_2)^{\frac{1}{2}}}{N_0} \leq M_1,
\end{equation*}
we only need to show
\begin{align*}
\textnormal{(I)} &  \sum_{N_0 \sim N_1 \sim N_2} \sum_{L_j \lesssim N_1}
\sum_{{\tiny{\substack{-M_0 \leq j_1,j_2 \leq M_0 -1\\|j_1 - j_2|\leq 16}}}} 
I({f^{\pm, c}}, {g_1^{-, M_0, j_1}}, {g_2^{+, M_0, j_2}}) 
\lesssim  \|f \|_{\ha{X}^{-s,\,1-b-\e}_{\pm,\,c}} \|g_1 \|_{\ha{X}^{s,\,b}_{-}}\|g_2 \|_{\ha{X}^{s,\,b}_{+}},\\
\textnormal{(I \hspace{-0.15cm}I)} &  \sum_{N_0 \sim N_1 \sim N_2} \sum_{L_j \lesssim N_1} \sum_{64 \leq A \leq M_0}
\sum_{{\tiny{\substack{-A \leq j_1,j_2 \leq A-1\\|j_1 - j_2|\leq 16}}}} 
I ({f^{\pm, c}}, {g_1^{-, A, j_1}}, {g_2^{+, A, j_2}}) 
\lesssim  \|f \|_{\ha{X}^{-s,\,1-b-\e}_{\pm,\,c}} \|g_1 \|_{\ha{X}^{s,\,b}_{-}}\|g_2 \|_{\ha{X}^{s,\,b}_{+}},\\
\textnormal{(I \hspace{-0.16cm}I \hspace{-0.16cm}I)} & \quad  \sum_{N_1} \sum_{1\ll N_0 \lesssim N_1 \sim N_2} \sum_{L_j \lesssim N_1}
\sum_{{\tiny{\substack{-M_1 \leq j_1,j_2 \leq M_1 -1\\|j_1 - j_2 \pm M_1|\leq 16}}}} 
I({f^{\pm, c}}, {g_1^{-, M_1, j_1}}, {g_2^{+, M_1, j_2}}) 
\lesssim  \|f \|_{\ha{X}^{-s,\,1-b-\e}_{\pm,\,c}} \|g_1 \|_{\ha{X}^{s,\,b}_{-}}\|g_2 \|_{\ha{X}^{s,\,b}_{+}},\\
\textnormal{(I \hspace{-0.18cm}V)} & \quad \sum_{N_1} \sum_{1\ll N_0 \lesssim N_1 \sim N_2} \sum_{L_j \lesssim N_1} \sum_{64 \leq A \leq M_1}
\sum_{{\tiny{\substack{-A \leq j_1,j_2 \leq A -1\\|j_1 - j_2 \pm A|\leq 16}}} } 
I({f^{\pm, c}}, {g_1^{-, A, j_1}}, {g_2^{+, A, j_2}}) 
\lesssim  \|f \|_{\ha{X}^{-s,\,1-b-\e}_{\pm,\,c}} \|g_1 \|_{\ha{X}^{s,\,b}_{-}}\|g_2 \|_{\ha{X}^{s,\,b}_{+}},
\end{align*}
where ${g_1^{-, A, j_1}} := g_1^-|_{{\mathfrak{D}}_{j_1}^A}$ and ${g_2^{+, A, j_2}} := g_2^+|_{{\mathfrak{D}}_{j_2}^A}$. We further simplify 
\textnormal{(I)}-\textnormal{(I \hspace{-0.18cm}V)}. From $\ell^2$ Cauchy-Schwarz inequality 
and $L_{\textnormal{max}}^{012} \lesssim N_1$, it suffices to show that there exists 
$0 < \e' < 1$ such that the following estimates hold;
\begin{align*}
\textnormal{(I)}' &  \ \sum_{{\tiny{\substack{-M_0 \leq j_1,j_2 \leq M_0 -1\\|j_1 - j_2|\leq 16}}}} 
I({f^{\pm, c}}, {g_1^{-, M_0, j_1}}, {g_2^{+, M_0, j_2}})  \lesssim  N_0^{s-\e'} (L_0 L_1 L_2)^{\frac{1}{2}} \|{f^{\pm, c}} \|_{{L_{\xi,\tau}^2}} \|{g_1^-}\|_{{L_{\xi,\tau}^2}} \|{g_2^+} \|_{{L_{\xi,\tau}^2}},\\
\textnormal{(I \hspace{-0.15cm}I)}' & \ \sum_{64 \leq A \leq M_0}
\sum_{{\tiny{\substack{-A \leq j_1,j_2 \leq A -1\\|j_1 - j_2|\leq 16}}}} 
I ({f^{\pm, c}}, {g_1^{-, A, j_1}}, {g_2^{+, A, j_2}}) \lesssim 
N_0^{s-\e'}  (L_0 L_1 L_2)^{\frac{1}{2}} \|{f^{\pm, c}} \|_{{L_{\xi,\tau}^2}}  \|{g_1^-}\|_{{L_{\xi,\tau}^2}} \|{g_2^+} \|_{{L_{\xi,\tau}^2}},\\
\textnormal{(I \hspace{-0.16cm}I \hspace{-0.16cm}I)}' & \ 
\sum_{{\tiny{\substack{-M_1 \leq j_1,j_2 \leq M_1 -1\\|j_1 - j_2 \pm M_1|\leq 16}}}} 
I({f^{\pm, c}}, {g_1^{-, M_1, j_1}}, {g_2^{+, M_1, j_2}}) \lesssim
N_0^{-s} N_1^{2s-\e'} (L_0 L_1 L_2)^{\frac{1}{2}} \| {f^{\pm, c}} \|_{{L_{\xi,\tau}^2}}  \|{g_1^-}\|_{{L_{\xi,\tau}^2}} \|{g_2^+} \|_{{L_{\xi,\tau}^2}},\\
\textnormal{(I \hspace{-0.18cm}V)}' & \ \sum_{64 \leq A \leq M_1}
\sum_{{\tiny{\substack{-A \leq j_1,j_2 \leq A -1\\|j_1 - j_2 \pm A|\leq 16}}} } 
I({f^{\pm, c}}, {g_1^{-, A, j_1}}, {g_2^{+, A, j_2}}) \lesssim
N_0^{-s} N_1^{2s-\e'} (L_0 L_1 L_2)^{\frac{1}{2}} \| {f^{\pm, c}} \|_{{L_{\xi,\tau}^2}}  \|{g_1^-}\|_{{L_{\xi,\tau}^2}} \|{g_2^+} \|_{{L_{\xi,\tau}^2}}.
\end{align*}
If $-3/4 < s$, \textnormal{(I)}$'$ is immediately established by using Proposition \ref{thm2.8}. 
\begin{align*}
&  \ \sum_{{\tiny{\substack{-M_0 \leq j_1,j_2 \leq M_0 -1\\|j_1 - j_2|\leq 16}}}} 
I({f^{\pm, c}}, {g_1^{-, M_0, j_1}}, {g_2^{+, M_0, j_2}}) \\
\sim & \sum_{{\tiny{\substack{-M_0 \leq j_1,j_2 \leq M_0 -1\\|j_1 - j_2|\leq 16}}}} N_0^{-1} 
\left| \int{ {f^{\pm, c}} (\tau, \xi) {g_1^{-, M_0, j_1}} (\tau_1, \xi_1) {g_2^{+, M_0, j_2}} (\tau_2, \xi_2) }d\tau_1 d\tau_2 d\xi_1 d \xi_2 \right|\\
\lesssim & \  N_0^{-\frac{3}{4}} (L_0 L_1)^{\frac{1}{2}}\|{f^{\pm, c}} \|_{{L_{\xi,\tau}^2}} 
\sum_{{\tiny{\substack{-M_0 \leq j_1,j_2 \leq M_0 -1\\|j_1 - j_2|\leq 16}}}}
\|{g_1^{-, M_0, j_1}} \|_{{L_{\xi,\tau}^2}} \|{g_2^{+, M_0, j_2}} \|_{{L_{\xi,\tau}^2}},\\
\lesssim & \ N_0^{-\frac{3}{4}} (L_0 L_1 L_2)^{\frac{1}{2}} \|{f^{\pm, c}} \|_{{L_{\xi,\tau}^2}}  \|{g_1^-}\|_{{L_{\xi,\tau}^2}} \|{g_2^+} \|_{{L_{\xi,\tau}^2}}.
\end{align*}
Next we prove \textnormal{(I \hspace{-0.15cm}I)}$'$. It follows from Proposition \ref{thm2.8} that
\begin{align*}
& \sum_{64 \leq A \leq M_0}
\sum_{{\tiny{\substack{-A \leq j_1,j_2 \leq A -1\\|j_1 - j_2|\leq 16}}}} 
I ({f^{\pm, c}}, {g_1^{-, A, j_1}}, {g_2^{+, A, j_2}})\\
\sim & \sum_{64 \leq A \leq M_0}
\sum_{{\tiny{\substack{-A \leq j_1,j_2 \leq A -1\\|j_1 - j_2|\leq 16}}}} 
N_0^{-1} 
\left| \int{ {f^{\pm, c}} (\tau, \xi) {g_1^{-, A, j_1}} (\tau_1, \xi_1) {g_2^{+, A, j_2}} (\tau_2, \xi_2) }d\tau_1 d\tau_2 d\xi_1 d \xi_2 \right|\\
\lesssim & \sum_{64 \leq A \leq M_0} N_0^{-1} A^{\frac{1}{2}} (L_0 L_1 L_2)^\frac{1}{2}
\| {f^{\pm, c}} \|_{{L_{\xi,\tau}^2}} \sum_{{\tiny{\substack{-A \leq j_1,j_2 \leq A -1\\|j_1 - j_2|\leq 16}}}}  \|{g_1^{-, A, j_1}} \|_{{L_{\xi,\tau}^2}} \|{g_2^{+, A, j_2}} \|_{{L_{\xi,\tau}^2}}\\
\lesssim & \ N_0^{-\frac{3}{4}} (L_0 L_1 L_2)^{\frac{1}{2}} \|{f^{\pm, c}} \|_{{L_{\xi,\tau}^2}}  \|{g_1^-}\|_{{L_{\xi,\tau}^2}} \|{g_2^+} \|_{{L_{\xi,\tau}^2}}.
\end{align*}
\textnormal{(I \hspace{-0.16cm}I \hspace{-0.16cm}I)}$'$ is verified as follows. By Lemma \ref{lem-modu1}, 
we have $N_0 \lesssim L_{\textnormal{max}}^{012}$. For the sake of simplicity, we here consider the case of $N_0 \lesssim L_0$. 
The other cases can be proved similarly. We deduce from Proposition \ref{thm2.1} and H\"{o}lder inequality that
\begin{align*}
& \sum_{{\tiny{\substack{-M_1 \leq j_1,j_2 \leq M_1 -1\\|j_1 - j_2 \pm M_1|\leq 16}}}} 
I({f^{\pm, c}}, {g_1^{-, M_1, j_1}}, {g_2^{+, M_1, j_2}})\\
\sim & N_0 N_1^{-2} \sum_{{\tiny{\substack{-M_1 \leq j_1,j_2 \leq M_1 -1\\|j_1 - j_2 \pm M_1|\leq 16}}}} 
\left| \int{ {f^{\pm, c}} (\tau, \xi) {g_1^{-, M_1, j_1}} (\tau_1, \xi_1) {g_2^{+, M_1, j_2}} (\tau_2, \xi_2) }d\tau_1 d\tau_2 d\xi_1 d \xi_2 \right|\\
\lesssim & N_0 N_1^{-2} \sum_{{\tiny{\substack{-M_1 \leq j_1,j_2 \leq M_1 -1\\|j_1 - j_2 \pm M_1|\leq 16}}}}  
\|\chi_{K^{\pm, c}_{N_0,L_0}} ({g_1^{-, M_1, j_1}} *  {g_2^{+, M_1, j_2}}) \|_{{L_{\xi,\tau}^2}} \|{f^{\pm, c}} \|_{{L_{\xi,\tau}^2}}\\
\lesssim & N_0N_1^{-2} N_1^{\frac{1}{4}} L_1^{\frac{1}{2}} L_2^{\frac{1}{4}} L_0^{\frac{1}{2}}
\|{f^{\pm, c}} \|_{{L_{\xi,\tau}^2}} \sum_{{\tiny{\substack{-M_1 \leq j_1,j_2 \leq M_1 -1\\|j_1 - j_2 \pm M_1|\leq 16}}}}  
 \|{g_1^{-, M_1, j_1}} \|_{{L_{\xi,\tau}^2}} \|{g_2^{+, M_1, j_2}} \|_{{L_{\xi,\tau}^2}}\\
\lesssim & N_0^{-s} N_1^{2s- \e'} (L_0 L_1 L_2)^{\frac{1}{2}}\|{f^{\pm, c}} \|_{{L_{\xi,\tau}^2}}  \|{g_1^-}\|_{{L_{\xi,\tau}^2}} \|{g_2^+} \|_{{L_{\xi,\tau}^2}}.
\end{align*}
Lastly, we prove \textnormal{(I \hspace{-0.18cm}V)}$'$. We use the two estimations depending on the relation between $N_0$ and $N_1$. 
Precisely, we utilize Proposition \ref{thm2.2} if $N_0^3 \lesssim N_1^2$, and if not so, we employ Proposition \ref{thm2.6}. 
We first assume $N_0^3 \lesssim N_1^2$. 
\begin{align*}
& \sum_{64 \leq A \leq M_1}
\sum_{{\tiny{\substack{-A \leq j_1,j_2 \leq A -1\\|j_1 - j_2 \pm A|\leq 16}}} } 
I({f^{\pm, c}}, {g_1^{-, A, j_1}}, {g_2^{+, A, j_2}})\\
\sim & N_0 N_1^{-2} \sum_{64 \leq A \leq M_1}
\sum_{{\tiny{\substack{-A \leq j_1,j_2 \leq A -1\\|j_1 - j_2 \pm A|\leq 16}}} } 
\left| \int{ {f^{\pm, c}} (\tau, \xi) {g_1^{-, A, j_1}} (\tau_1, \xi_1) {g_2^{+, A, j_2}} (\tau_2, \xi_2) }d\tau_1 d\tau_2 d\xi_1 d \xi_2 \right|\\
\lesssim & N_0 N_1^{-2} \sum_{64 \leq A \leq M_1}
\sum_{{\tiny{\substack{-A \leq j_1,j_2 \leq A -1\\|j_1 - j_2 \pm A|\leq 16}}} } 
\|\chi_{{K^{-}_{N_1,L_1}}} ({f^{\pm, c}} *  g_{2,-}^{+,A,j_2}) \|_{{L_{\xi,\tau}^2}} \|{g_1^{-, A, j_1}} \|_{{L_{\xi,\tau}^2}}\\
\lesssim & N_0N_1^{-2} N_0^{\frac{1}{2}}(L_0L_2)^{\frac{1}{2}} \sum_{64 \leq A \leq M_1}
\|{f^{\pm, c}} \|_{{L_{\xi,\tau}^2}}
\sum_{{\tiny{\substack{-A \leq j_1,j_2 \leq A -1\\|j_1 - j_2 \pm A|\leq 16}}} } 
 \|{g_1^{-, A, j_1}} \|_{{L_{\xi,\tau}^2}} \|{g_2^{+, A, j_2}} \|_{{L_{\xi,\tau}^2}}\\
\lesssim & N_0^{\frac{3}{2}} N_1^{-2+\e'} (L_0 L_1 L_2)^{\frac{1}{2}}\|{f^{\pm, c}} \|_{{L_{\xi,\tau}^2}}  \|{g_1^-}\|_{{L_{\xi,\tau}^2}} \|{g_2^+} \|_{{L_{\xi,\tau}^2}}\\
\lesssim & N_0^{\frac{3}{2}} N_1^{-2-2 s + 2 \e'} N_1^{2 s -\e'} (L_0 L_1 L_2)^{\frac{1}{2}}\|{f^{\pm, c}} \|_{{L_{\xi,\tau}^2}} 
 \|{g_1^-}\|_{{L_{\xi,\tau}^2}} \|{g_2^+} \|_{{L_{\xi,\tau}^2}}\\
\lesssim & N_0^{-s +3 \left( \e' - \frac{2}{3} \left(s + \frac{3}{4} \right) \right)}N_1^{2 s -\e'}
 (L_0 L_1 L_2)^{\frac{1}{2}}\|{f^{\pm, c}} \|_{{L_{\xi,\tau}^2}} \|{g_1^-}\|_{{L_{\xi,\tau}^2}} \|{g_2^+} \|_{{L_{\xi,\tau}^2}}.
\end{align*}
If $0 < \e' \leq \frac{2}{3} \left(s + \frac{3}{4} \right)$, this completes \textnormal{(I \hspace{-0.18cm}V)}$'$. 
We next assume $N_0^3 \gtrsim N_1^2$. 
From Proposition \ref{thm2.6} and $M_1 \sim N_1/N_0$, we observe that
\begin{align*}
& \sum_{64 \leq A \leq {M_1}}
\sum_{{\tiny{\substack{-A \leq j_1,j_2 \leq A -1\\|j_1 - j_2 \pm A|\leq 16}}} } 
I({f^{\pm, c}}, {g_1^{-, A, j_1}}, {g_2^{+, A, j_2}})\\
\sim & N_0 N_1^{-2} \sum_{64 \leq A \leq {M_1}}
\sum_{{\tiny{\substack{-A \leq j_1,j_2 \leq A -1\\|j_1 - j_2 \pm A|\leq 16}}} } 
\left| \int{ {f^{\pm, c}} (\tau, \xi) {g_1^{-, A, j_1}} (\tau_1, \xi_1) {g_2^{+, A, j_2}} (\tau_2, \xi_2) }d\tau_1 d\tau_2 d\xi_1 d \xi_2 \right|\\
\lesssim & N_0 N_1^{-2} \sum_{64 \leq A \leq {M_1}} A^{\frac{7}{8}} (L_0 L_1 L_2)^{\frac{1}{2}} 
\|{f^{\pm, c}} \|_{{L_{\xi,\tau}^2}}
\sum_{{\tiny{\substack{-A \leq j_1,j_2 \leq A -1\\|j_1 - j_2 \pm A|\leq 16}}} } 
 \|{g_1^{-, A, j_1}} \|_{{L_{\xi,\tau}^2}} \|{g_2^{+, A, j_2}} \|_{{L_{\xi,\tau}^2}}\\
\lesssim & N_0 N_1^{-2} N_1^{\frac{7}{8}} N_0^{- \frac{7}{8}}
(L_0 L_1 L_2)^{\frac{1}{2}}\|{f^{\pm, c}} \|_{{L_{\xi,\tau}^2}} \|{g_1^-}\|_{{L_{\xi,\tau}^2}} \|{g_2^+} \|_{{L_{\xi,\tau}^2}}\\
\lesssim & N_0^{\frac{1}{8}} N_1^{-\frac{9}{8}}(L_0 L_1 L_2)^{\frac{1}{2}}\|{f^{\pm, c}} \|_{{L_{\xi,\tau}^2}} \|{g_1^-}\|_{{L_{\xi,\tau}^2}} \|{g_2^+} \|_{{L_{\xi,\tau}^2}}\\
\lesssim & N_0^{-s} N_0^{s+\frac{1}{8}} N_1^{-\frac{9}{8}}(L_0 L_1 L_2)^{\frac{1}{2}}\|{f^{\pm, c}} \|_{{L_{\xi,\tau}^2}} \|{g_1^-}\|_{{L_{\xi,\tau}^2}} \|{g_2^+} \|_{{L_{\xi,\tau}^2}}\\
\lesssim & N_0^{-s} N_1^{2 s - \left( \frac{4}{3} s + \frac{25}{24} \right)}
(L_0 L_1 L_2)^{\frac{1}{2}}\|{f^{\pm, c}} \|_{{L_{\xi,\tau}^2}} \|{g_1^-}\|_{{L_{\xi,\tau}^2}} \|{g_2^+} \|_{{L_{\xi,\tau}^2}}.
\end{align*}
This completes the proof of \textnormal{(I \hspace{-0.18cm}V)}$'$. 
\end{proof}
\section{Negative result}
In this section, we establish Theorem \ref{not-c2}. For convenience, we restate the Theorem \ref{not-c2}.
\begin{thm}\label{not-c2-2}
Let $d=2$, $0<c<1$ and $s < - \frac{3}{4}$. Then for any $T>0$, the data-to-solution map 
$ (u_0, u_1, n_0, n_1) \mapsto (u, n)$ 
of \eqref{KGZ}, as a map from the unit ball in 
$ H^{s+1} \cross H^s \cross H^s \cross H^{s-1}$ to 
$C([0,T]; H^{s+1}) \cap C^1([0,T];H^s) \cross C([0,T]; H^s) \cap  
C^1 ([0,T];H^{s-1})$ fails to be $C^2$.
\end{thm}
\begin{proof}
By the same argument as in the proof of Theorem 1.4 in \cite{Ho}, it suffices to prove that for every $C>0$, there exist real-valued functions $u_0 \in H^{s+1}$ and 
$n_0 \in H^s$ such that
\begin{align}
\sup_{0 < t \leq T}
\left\| \int_0^t{ \frac{\sin{((t-t') \LR{\nabla})}}{\LR{\nabla}} \left(  
(\cos{(t' \LR{\nabla})}u_0) (\cos{(t'|\nabla|)} n_0) \right)}dt' \right\|_{H^{s+1}} & \notag \\
 \geq C \|u_0\|_{H^{s+1}} & \|n_0\|_{{H}^s}.\label{c2-comp}
\end{align}
Let $N \gg 1$, we define the sets $D_1$, $D_2$, $D_3 \subset \R^2$ as
\begin{align*}
D_1 & := [N, \, N+1] \cross [-N^{\frac{1}{2}}, \, N^{\frac{1}{2}}] \cup  
 [-N -1, \, -N] \cross [-N^{\frac{1}{2}}, \, N^{\frac{1}{2}}],\\
D_2 & := \left[\frac{2N}{1-c}, \, \frac{2N}{1-c}+1\right] 
\cross [-N^{\frac{1}{2}}, \, N^{\frac{1}{2}}] \cup 
 \left[- \frac{2N}{1-c} -1, \, - \frac{2N}{1-c} \right] \cross [-N^{\frac{1}{2}}, \, N^{\frac{1}{2}}],\\
D_3 & := \left[\frac{1+c}{1-c}N -1, \, \frac{1+c}{1-c}N+1\right] 
\cross [-N^{\frac{1}{2}}, \, N^{\frac{1}{2}}] \cup 
 \left[- \frac{1+c}{1-c}N -1, \, - \frac{1+c}{1-c} N +1\right] \cross [-N^{\frac{1}{2}}, \, N^{\frac{1}{2}}].\\
\end{align*}
We set $(u_{N,0}, n_{N,0})$ as
\begin{equation*}
 (\F_x u_{N,0}) (\xi) := N^{-s-\frac{5}{4}}\chi_{D_1} (\xi), \quad 
(\F_x n_{N,0}) (\xi) := N^{-s-\frac{1}{4}}\chi_{D_2}(\xi).
\end{equation*}
We easily verify that $u_{N,0}$, $n_{N,0}$ are real-valued and $\|u_{N,0} \|_{H^{s+1}} \sim 1$, $\| n_{N,0} \|_{{H}^s} \sim 1$. A simple calculation gives
\begin{align}
\F_x & \left( \int_0^t{ \frac{\sin{((t-t') \LR{\nabla})}}{\LR{\nabla}} \left(  
(\cos{(t' \LR{\nabla})}u_{N,0}) (\cos{(t'|\nabla|)} n_{N,0}) \right)}dt' \right)\chi_{D_3}(\xi) \notag\\
\sim & N^{-2 s - \frac{5}{2}} \int_0^t \int_{\R^2} \left( e^{-i(t-t') \LR{\xi}} - e^{i(t-t') \LR{\xi}} \right) 
\chi_{D_3}(\xi) \cross \notag \\ 
&  {\left(  
((e^{-ict' \LR{\xi_1}} - e^{ict' \LR{\xi_1}})) \chi_{D_1}(\xi_1) ( e^{-it'|\xi-\xi_1|} - e^{it'|\xi -\xi_1|} n_{N,0})
\chi_{D_2} (\xi-\xi_1) \right)}d\xi_1 dt' \notag \\
\sim & N^{-2s - 2}\left( e^{-it \frac{1+c}{1-c} N} - e^{it \frac{1+c}{1-c} N} \right) t \, \chi_{D_3}(\xi) 
+ N^{-2s -2} \mathcal{O}(t^2)\chi_{D_3} (\xi), \label{c2-first}
\end{align}
for $N^{-1} \leq t \ll 1$. For any sufficient small $t>0$, we can choose $N \gg 1$ such that
\begin{equation*}
N^{-1} \leq t \ll 1 \quad \textnormal{and} \quad \left| e^{-it \frac{1+c}{1-c} N} - e^{it \frac{1+c}{1-c} N} \right| 
\geq \frac{1}{2}.
\end{equation*}
Thus we can choose $0<t<T$, $N \gg 1$ such that the first term of \eqref{c2-first} dominates the second term.
\begin{align*}
& \left\| \int_0^t{ \frac{\sin{((t-t') \LR{\nabla})}}{\LR{\nabla}} \left(  
(\cos{(t' \LR{\nabla})}u_{N,0}) (\cos{(t'|\nabla|)} n_{N,0}) \right)}dt' \right\|_{H^{s+1}}\\
\geq & \left\| \LR{\xi}^{s+1} \chi_{D_3} \F_x \left(
 \int_0^t{ \frac{\sin{((t-t') \LR{\nabla})}}{\LR{\nabla}} \left(  
(\cos{(t' \LR{\nabla})}u_{N,0}) (\cos{(t'|\nabla|)} n_{N,0}) \right)}dt' \right) \right\|_{L^2}\\
\sim & N^{-s-1} \| \chi_{D_3} \|_{L^2} \sim N^{-s -\frac{3}{4}}.
\end{align*}
This completes the proof of \eqref{c2-comp}.
\end{proof}

\section*{Acknowledgments} The author appreciates Professor M. Sugimoto, Professor K. Tsugawa and Professor I. Kato for giving many useful advices to the author. The author is supported by Grant-in-Aid for JSPS Research Fellow 16J11453.


\end{document}